\pdfoutput=1
\documentclass{amsart}
\usepackage[utf8]{inputenc}
\usepackage{amsthm,amsmath,amssymb, amscd}
\usepackage[colorlinks=true,allcolors=blue!80!black]{hyperref}
\usepackage[parfill]{parskip}  
\usepackage{graphics}
\usepackage{tikz, tikz-cd}
\usetikzlibrary{matrix, snakes, patterns, shadows.blur, backgrounds}
\usepackage[all]{xy}
\usepackage{enumerate}
\usepackage{bbm}
\usepackage{caption}
\usepackage{subcaption}
\DeclareMathAlphabet\mathbfcal{OMS}{cmsy}{b}{n}

\usepackage{todonotes}

\theoremstyle{plain}
\newtheorem{thm}{Theorem}[section]
\newtheorem{lem}[thm]{Lemma}
\newtheorem{prop}[thm]{Proposition}
\newtheorem{cor}[thm]{Corollary}

\newtheorem*{thm*}{Theorem}
\newtheorem{abcthm}{Theorem}
\newtheorem{abccor}[abcthm]{Corollary}

\theoremstyle{definition}
\newtheorem{defn}[thm]{Definition}

\newtheorem*{exam*}{Example}
\newtheorem{exam}[thm]{Example}
\newtheorem{rem}[thm]{Remark}

\newcommand{\fakeenv}{} 
\newenvironment{restate}[2]  
{
  \renewcommand{\fakeenv}{#2} 
  \theoremstyle{plain}
  \newtheorem*{\fakeenv}{#1~\ref{#2}} 
  \begin{\fakeenv}  
}
{
  \end{\fakeenv}
}
\def\Z{\mathbb Z}
\def\C{\mathbb C}
\def\R{\mathbb R}

\def\N{\mathbb N}

\def\B{\mathbb B}

\def\p{\partial}

\def\S{\Sigma}

\def\ul{\underline}

\def\bt{\bullet}

\def\ESS{\operatorname{ESS}}
\def\PESS{\operatorname{PESS}}
\DeclareMathOperator\Int{Int}

\DeclareMathOperator\im{im}

\DeclareMathOperator\Id{Id}
\DeclareMathOperator\GL{GL}

\DeclareMathOperator\Aut{Aut}
\DeclareMathOperator\Inn{Inn}

\DeclareMathOperator\id{Id}

\DeclareMathOperator\pconf{Conf}
\DeclareMathOperator\emb{Emb}
\DeclareMathOperator\link{\mathcal{E}}
\DeclareMathOperator\plink{\mathcal{PE}}
\DeclareMathOperator\pSep{PSep}

\DeclareMathOperator\Diff{Diff}
\DeclareMathOperator\Homeo{Homeo}

\DeclareMathOperator\lk{lk}

\DeclareMathOperator\Sub{Sub}
\DeclareMathOperator\remb{REmb}

\DeclareMathOperator{\Desc}{Desc}
\DeclareMathOperator{\FR}{FR}
\DeclareMathOperator{\st}{star}
\DeclareMathOperator{\Disk}{Disk}
\DeclareMathOperator{\In}{In}
\DeclareMathOperator{\SO}{SO}
\def\Sep{\operatorname{Sep}}
\def\RSep{\operatorname{RSep}}
\def\SAut{\Sigma\operatorname{Aut}}

\def\TSep{\operatorname{Sep}^{\pitchfork}}
\def\TSeprho{\operatorname{Sep}(\rho)^{\pitchfork}}
\def\uemb{\Sub}
\def\uremb{R\Sub}
\def\FRL{\mathcal{FR}(L)}

\definecolor{champagne}{rgb}{0.97, 0.91, 0.81}
\definecolor{burlywood}{rgb}{0.55, 0.71, 0.0}
\definecolor{cfblue}{rgb}{0.39, 0.58, 0.93}
\definecolor{amber}{rgb}{1.0, 0.75, 0.0}

\title{Embedding spaces of split links}

\author{Rachael Boyd}
\address{School of Mathematics and Statistics, University of Glasgow, Glasgow G12 8QQ, UK}
\email{rachael.boyd@glasgow.ac.uk}
\urladdr{https://www.maths.gla.ac.uk/~rboyd/} 

\author{Corey Bregman}
\address{Department of Mathematics, Tufts University, Medford, MA 02155, USA}
\email{corey.bregman@tufts.edu}
\urladdr{https://sites.google.com/view/cbregman}

\subjclass[2010]{58D10, 55P15, 55U10, 20F34 (primary), 57M07, 20F36 (secondary)}

\begin{document}
\maketitle
\begin{abstract}
    We study the homotopy type of the space~$\link(L)$ of unparametrised embeddings of a split link~$L=L_1\sqcup \ldots \sqcup L_n$ in~$\R^3$.  Our main result is a simple description of the fundamental group, or motion group, of~$\link(L)$, and we extend this to a description of the motion group of embeddings in~$S^3$. The main tool we build is a semi-simplicial space of separating systems, which we show is homotopy equivalent to~$\link(L)$. This combinatorial object provides a gateway to studying the homotopy type of $\link(L)$ via the homotopy type of the spaces~$\link(L_i)$.
\end{abstract}
\section{Introduction}

Let $L\subset \R^3$ be a one-dimensional submanifold, \emph{i.e.} a link. In this paper we study the homotopy type of the configuration space of all links isotopic to $L$. We say~$L$ is \emph{split} if it can be written as a disjoint union~$L=L_1\sqcup\cdots\sqcup L_n$, where the~$L_i\neq\emptyset$ can be contained in disjoint balls. We assume each of the $L_i$ are not themselves split, and call each sublink~$L_i$ a \emph{piece} of~$L$, so that $n$ is the number of pieces. If the total number of circle components of~$L$ is~$m$, let~$\link(L)$ denote the connected component of
\[
\uemb(\sqcup_m S^1, \R^3)=\emb(\sqcup_m S^1, \R^3)/\Diff(\sqcup_m S^1)
\]
containing the link~$L$. We quotient by the free action of~$\Diff(\sqcup_m S^1)$ in order to consider unparametrised smooth embeddings, hence `$\uemb$' stands for `submanifold'. 

We are therefore interested in the homotopy type of the space $\link(L)$, and our main result is a general description of~$\pi_1(\link(L))$ for any split link~$L$. We call this group the \emph{motion group of $L$}, and exhibit it in terms of~$\pi_1(\link(L_i))$ and~$\pi_1(\R^3\setminus L_i)$ for each piece~$L_i$. We will discuss this result first, before introducing the more general theory we built to study the homotopy groups of~$\link(L)$.

\subsection{Computation of the fundamental group}

We use our theory (Theorem~\ref{abcthm - sep equiv}) to decompose the motion group $\pi_1(\link(L))$ (with basepoint $L$ omitted from the notation) into an iterated semidirect product of three groups:
\begin{enumerate}
    \item \textbf{$\mathbf{P_L}$:} The symmetric group $S_n$ acts on the set of~$n$ link pieces by permutation, and~$P_L$ is the subgroup of permutations which only permute link pieces~$L_i$ and~$L_j$ when they are isotopic in~$\R^3$.
    \item \textbf{$\mathbf{G_L}$:} is the direct product of the motion groups of the link pieces, \emph{i.e.,} $$G_L=\prod_{i=1}^n \pi_1(\link(L_i)).$$
    \item \textbf{$\mathbfcal{FR}\mathbf{(L)}
    $:} Let~$H_i=\pi_1(\R^3\setminus L_i)$ for each piece~$L_i$ of~$L$ and let $H_L=\pi_1(\R^3\setminus L)\cong H_1*\cdots *H_n$. We define $\FRL$ to be the subgroup of~$\Aut(H_L)$ generated by automorphisms that conjugate one factor $H_j$ by an element $g\in H_i$ where $i\neq j$ and call it the  \emph{Fouxe-Rabinovitch} group of $L$ after \cite{FouxeRabinovitch1,FouxeRabinovitch2} (see Section \ref{sec:Dahm-Fouxe-Rabinovitch} for background on these groups).
\end{enumerate}

\begin{abcthm} \label{abcthm - motion group}
The motion group $\pi_1(\link(L))$ is isomorphic to $(\FRL\rtimes G_L)\rtimes P_L$. 
\end{abcthm}

The action of $P_L$ and $G_L$ in this iterated semidirect product decomposition may be described as follows. $P_L$ acts by permuting the link pieces, which has the effect of permuting the factors in $G_L=\prod_{i=1}^n \pi_1(\link(L_i)),$ as well as the conjugating elements~$g\in H_i$ and factors~$H_j$ in $\FRL$. For the inner semidirect product, each factor $\pi_1(\link(L_i))$ acts on $H_i=\pi_1(\R^3\setminus L_i)$  (via the Dahm homomorphism~\cite{Dahm, Wattenberg1972}, see Section \ref{sec:Dahm-Fouxe-Rabinovitch} for a definition). If an element $y\in\pi_1(\link(L_i))$ induces an automorphism $\varphi$ of $H_i$, then $y$ acts on $\FRL$ by sending the conjugation of~$H_j$ by $g\in H_i$ to the conjugation of~$H_j$ by $\varphi(g)$.

We prove this theorem by first exhibiting motions in~$\pi_1(\link(L))$ that generate the various subgroups. We then construct two split exact sequences, the first of which splits off~$P_L$. Using Theorem~\ref{abcthm - sep equiv}, we show that~$\FRL$ is the kernel of the second, which splits off $G_L$. We remark that the Fouxe-Rabinovitch group of a free product is finitely presented as long as each of the factors are.  In particular, it follows from Theorem \ref{abcthm - motion group} that $\pi_1(\link(L))$ is finitely presented as long as $\pi_1(\link(L_i))$ is for each piece~$L_i$.

Let us illustrate this semidirect product decomposition in the special case when~$L=U_n$ is the $n$-component unlink. The motion group in this case was first computed by Wattenberg (following Dahm) \cite{Wattenberg1972, Dahm}, who identified $\pi_1(\link(U_n))$ with the symmetric automorphism group~$\SAut(F_n)$  of the free group $F_n$. This is the subgroup of $\Aut(F_n)$ whose elements send generators of $F_n$ to conjugates of generators or their inverses. This identification can be understood as follows:~$F_n$ is the fundamental group of~$\R^3\setminus U_n$, and a motion in $\link(U_n)$ starting and ending at~$U_n$ acts as an automorphism of this fundamental group. 

A presentation for $\SAut(F_n)$ was found by McCool \cite{McCool} and Savushkina \cite{Savushkina1996}, and later recovered by Brendle and Hatcher \cite{BrendleHatcher} in the context of motion groups. The decomposition given by Theorem \ref{abcthm - motion group} amounts to a description of $\SAut(F_n)$ as the semidirect product $\SAut(F_n)\cong (\mathcal{FR}(U_n)\rtimes (\Z/2\Z)^n)\rtimes S_n$.  Let us relate these 3 groups to motions of $U_n$:
\begin{enumerate}
    \item View $U_n$ is a split link with~$n$ unknot pieces. Then since any two pieces are isotopic, there is a subgroup of the motion group isomorphic to the symmetric group~$S_n$ given by motions which permute the pieces.
    \item Since each unknot is unparametrised,~$\pi_1(\link(L_i))=\pi_1(\link(U))=\Z/2\Z$. A generator for this is a $180^\circ$ rotation about a diameter of a round unknot. Since we can realise these motions disjointly, we get a copy of~$(\Z/2\Z)^n$ in~$\pi_1(\link(U_n))$.
    \item  In this case, the Fouxe-Rabinovitch group is generated by motions where one unknot piece shrinks, moves through another, and returns to its original position. This corresponds the pure symmetric automorphism group of $F_n$ in which every generator is sent to a conjugate of itself.
\end{enumerate}

\subsection{Embeddings in \texorpdfstring{$S^3$}{the 3-sphere.}} Thinking of $S^3$ as the 1-point compactification of $\R^3$, there is a fibre sequence relating $\link(L)$ to $\uemb(L,S^3)$, which we define to be the path component of $\emb(\sqcup_m S^1,S^3)/\Diff(\sqcup_m S^1)$ containing $L$. We explore this relationship in Section \ref{section - the fundamental group}, and extend Theorem \ref{abcthm - motion group} to this setting. In fact, the fundamental group of $\uemb(L,S^3)$ is a quotient of $\pi_1(\link(L))$, by inner automorphisms of $\pi_1(S^3\setminus L)\cong\pi_1(\R^3\setminus L)$:
\begin{abccor}\label{abccor - motion group S3}Let $H_L=\pi_1(S^3\setminus L)$. Then 
\[\pi_1(\uemb(L,S^3))\cong \left((\FRL\rtimes G_L)/\Inn(H_L)\right)\rtimes P_L.\]
where $\Inn(H_L)$ is the group of inner automorphisms of $H_L$.
\end{abccor}

\subsection{Separating systems}
Brendle and Hatcher \cite{BrendleHatcher} showed that~$\link(U_n)$ (which they call the \emph{configuration space} of the link) is homotopy equivalent to the embedding space of \emph{round} unlinks, where round embeddings are such that each embedded~$S^1$ bounds a Euclidean disk in some affine plane~$\R^2\subset \R^3$. 
They consider collections of spheres, called 
\emph{separating systems}, which split the unlink into unknot components.  

We build a semi-simplicial space of {separating systems} for~$\link(L)$, which parametrises all of the ways of splitting links in~$\link(L)$ into pieces, and is topologised via the Whitney topology for smooth embedding spaces. As in Brendle--Hatcher, this reduces the codimension two problem of studying any homotopy group of~$\link(L)$ into
\begin{enumerate}[(i)]
    \item a codimension one problem - studying the space of separating $2$-spheres; 
    \item an easier codimension two problem - studying the corresponding homotopy group of~$\link(L_i)$ for each piece~$L_i$.
\end{enumerate}

Using this tool, for each~$k$ and each homotopy class in~$\pi_k(\link(L))$, we are able to exhibit a representative with certain extra structure. It is precisely this extra structure which enables us to compute the fundamental group~$\pi_1(\link(L))$ in Theorem~\ref{abcthm - motion group}.

A separating system for a link~$\rho\in \link(L)$ is a disjoint union of unparametrised embedded spheres in~$\R^3\setminus \rho$, which we denote by~$\Sigma$, such that each connected component of~$\R^3\setminus \S$ contains at most one piece of~$\rho$. Some examples of separating systems for the 3 component unlink are shown below.

\begin{center}
	\begin{tikzpicture}[framed, scale=0.6]
	\shade[ball color = cyan!40, opacity = 0.2] (0,0) circle (1.5cm);
	\shade[ball color = cyan!40, opacity = 0.4] (-.4,-.35) circle (0.35cm);
	\shade[ball color = cyan!40, opacity = 0.4] (2.5,2) circle (0.75cm);
	\draw (0,.5) circle [x radius=.3cm, y radius=.35cm];
	\draw[rotate=18] (2.9,1.1)  circle [x radius=.55cm, y radius=.25cm];
	\draw[rotate=40] (-0.65,-.05) circle [x radius=.1cm, y radius=.2cm];
	\shade[ball color = pink!40, opacity = 0] (1.5,.3) circle (2.75cm);
	\end{tikzpicture}
	\qquad
	\begin{tikzpicture}[framed, scale=0.6]
	\shade[ball color = cyan!40, opacity = 0] (0,0) circle (1.5cm);
	\shade[ball color = pink!60, opacity = 0.4] (.27,0) circle (1.2cm);
	\shade[ball color = pink!40, opacity = 0.4] (-.4,-.35) circle (0.35cm);
	\shade[ball color = pink!40, opacity = 0.2] (1.5,.3) circle (2.75cm);
	\draw (0,.5) circle [x radius=.3cm, y radius=.35cm];
	\draw[rotate=18] (2.9,1.1)  circle [x radius=.55cm, y radius=.25cm];
	\draw[rotate=40] (-0.65,-.05) circle [x radius=.1cm, y radius=.2cm];
	\end{tikzpicture}
\end{center}

We build a semi-simplicial space~$\Sep_\bt$ of separating systems for~$\link(L)$, where the $0$-simplices~$\Sep_0$ is the space of pairs~$(\S,\rho)$ such that~$\rho \in \link(L)$ and~$\S$ is a separating system for~$\rho$. The $p$-simplices~$\Sep_p \subset (\Sep_0)^{p+1}$ is the space of $(p+1)$ disjoint separating systems for the same link~$\rho$.

There is a natural augmentation map $\varepsilon_\bt\colon\Sep_\bt{\rightarrow} \link(L)$ which forgets the separating systems. Our theorem concerns the induced map on the geometric realisation~$|\Sep_\bt|$.

\begin{abcthm}\label{abcthm - sep equiv}
The map $|\varepsilon_\bt|\colon|\Sep_\bt|{\rightarrow} \link(L)$
is a homotopy equivalence.
\end{abcthm}

This result is used in the proof of Theorem~\ref{abcthm - motion group}, but we can also apply this theorem to get a rigid structure on the homotopy classes of $\link(L)$, which we do in Appendix~\ref{section - implications for homotopy groups}. We show the theorem implies that each homotopy class in~$\pi_k(\link(L))$ has a representative~$f\colon S^k\to \link(L)$ such that the image of~$f$ exhibits what we call a \emph{compatible separating triangulation}. This is a triangulation of the~sphere $S^k$ such that on the image of each simplex under~$f$, there exists a separating system for the embedded link which changes by isotopy as the link does. Moreover on faces where two or more simplices meet, all separating systems exist simultaneously and disjointly.

In fact we go further, and prove that in each homotopy class there is a representative with compatible separating triangulation such that in addition all separating spheres are \emph{round} \emph{i.e.,}~bound Euclidean balls in~$\R^3$.

\subsection{Related work}

Previous work on embedding spaces of links in dimension 3 has been primarily concerned with the case when~$m=1$, and ~$L$ is a knot. Indeed, an iterative formula for the homotopy type of the moduli  space of a \emph{parametrised} knot $K$ in $S^3$---$\emb_K(S^1,S^3)$---has been computed in work of Hatcher \cite{Hatcher, Hatcher2002}, Budney \cite{Budney, Budney2010, BudneySplicing} and Budney--Cohen \cite{BudneyCohen}. Central to their work is the associated space of \emph{long knots}:  embeddings of~$\mathbb{R}^1$ in~$\mathbb{R}^3$ which agree with the standard embedding outside the unit ball. Budney gives two decompositions of the moduli spaces of parametrised \emph{long knots} in $\R^3$, via the \emph{little 2-discs} operad \cite{Budney} and \emph{satellite} decomposition \cite{Budney2010,BudneySplicing}.

For embedding spaces of links, much less is known. Early approaches focused on computing~$\pi_1(\link(L))$ for specific links. This group, which we call the \emph{motion group}, is isomorphic to the relative homotopy group $\pi_1(\Diff(\R^3), \Diff(\R^3,L), \Id)$ \cite{Wattenberg1972}. In fact, this isomorphism holds for the motion group of any submanifold $V$ of a manifold~$M$:
\[ \pi_1(\Diff(M), \Diff(M,V),\Id)\cong \pi_1(\Sub(V,M),V)
,\]
by applying a standard result about fibrations and relative homotopy groups \cite[7.2.9]{Spanier} to Palais' fiber sequence (see, e.g.~\cite[Corollary 2.5]{BoydBregmanSteinebrunner24})
\[
\Diff(M,V) \to \Diff(M) \to \Sub(V,M),
\]
where $\Sub(V,M)$ denotes the unparametrised embedding space. Wattenberg \cite{Wattenberg1972} calls this group the \emph{smooth motion group} of $V$ in $M$, in comparison with the \emph{topological motion group} $\pi_1(\Homeo(M), \Homeo(M,V); \id)$ studied by e.g.~Dahm \cite{Dahm} and Goldsmith~\cite{Goldsmith1, Goldsmith}. For $V\subset M$ a compact surface in a 3-manifold, the smooth and topological motion groups can be shown to be isomorphic, following \cite{Cerf}. However for links in 3-manifolds there is no analogous proof in the literature, though the authors expect it to be true.

Goldsmith \cite{Goldsmith1}, citing work of Dahm \cite{Dahm} for the unlink case, computed the topological motion group of the $n$-component unlink~$U_n$, and torus links~\cite{Goldsmith}, and Wattenberg made the same computation for the smooth motion group of the unlink. More recently, Bellingeri--Bodin computed the motion group of the \emph{necklace link} in $\R^3$ \cite{BellingeriBodin2016} and Damiani and Kamada computed the motion group for the round unparameterised embedding space of the disjoint union of a Hopf link and an unknot \cite{DamianiKamada}. 

For parametrised links, Burke and Koytcheff \cite{BurkeKoytcheff2015}  introduce a coloured operad which extends Budney's work on knot splicing \cite{BudneySplicing} to the notion of link \emph{infection}. They study the space of two component string links using this operad. 
More recently, Havens and Koytcheff \cite{HavensKoytcheff2021} studied the homotopy type of the spaces of framed and unframed parametrised links in $S^3$, modulo the action of $\SO(4)$ on $S^3$ by rotations. By fixing one component $K$ of a link, and considering the other components in the complement of  a tubular neighbourhood $S^3\setminus \nu(K)$, they prove results about spaces of parametrised knots in $S^1\times D^2$, and $T^2\times I$, using methods which extend \cite{Budney2010}. 

We have restricted ourselves to links in dimension 3 for this discussion, but note here that there are many interesting results on moduli spaces of higher dimensional links.

In the past 50 years there have been a series of announcements \cite{CesardeSaRourke}, papers \cite{HendriksLaudenbach1984,HendriksMcCullough}, and unfinished manuscripts \cite{Hatcherwrongway}, studying the homotopy type of $\Diff(M)$ for a reducible 3-manifold $M$, in terms of fiber sequences. One of the main innovations in these works is the consideration of a configuration space parametrising decompositions of $M$ into irreducible pieces. If one considers $M$ to be a link complement, this is analogous to our space of separating systems. In work with Steinebrunner, we develop these ideas further to study the homotopy type of $B \Diff(M)$ for reducible 3-manifolds \cite{BoydBregmanSteinebrunner24}.

\subsection{Discussion and future directions}
The proof of Theorem \ref{abcthm - motion group} in Section \ref{section - the fundamental group} suggests a strategy for computing the higher homotopy groups of $\link(L)$. In fact, we expect the higher homotopy groups $\pi_k(\link(L))$ to be built out of the classes in $\pi_k(\link(L_i))$ together with classes in $\pi_k$ of the configuration space of $p$ points in $\R^3\setminus L_i$. If each piece is a knot, then the  groups $\pi_k(\link(L_i))$ are understood, as discussed in the previous section, although concrete computations in the literature are sparse. Little is known in the case that the piece is a link, except in the case of topological motions of torus links by Goldsmith~\cite{Goldsmith}. In upcoming work, we will prove that the unparametrised embedding space of a Hopf link,~$\link(HL)$, is homotopy equivalent to the round subspace~$R\link(HL
)$, where each link component bounds a Euclidean disk in some affine~$\R^2\subset \R^3$. This has the homotopy type of an explicit finite-dimensional manifold, which we are able to describe.

Homological stability for unparametrised embeddings of unlinks in 3-manifolds was proved by Kupers in \cite{Kupers}. We expect to be able to use our techniques, and Theorem~\ref{abcthm - sep equiv} in particular, to study the homology of~$\link(L)$ and more general homological stability phenomena.

Finally, we point out an application of our work to diffeomorphisms of 4-manifolds. Regarding $L$ as the locus of a collection of attaching spheres of 2-handles in $S^3$, we can think of a point of $\uemb(L,S^3)$, together with framing data, as describing a Kirby diagram of a 4-manifold $M$. If $\partial M=S^3$, we can cap off $M$ by adding a 4-handle to obtain a closed 4-manifold $\widehat{M}$. 

A motion of~$L$, representing an element of $\pi_1(\uemb(L,S^3))$, then gives rise to a diffeomorphism of~$M$ supported in a collar of the~$0$-handle. Since~$\Diff_0(S^3)$ is connected, this extends to a diffeomorphism of~${M}\cup(S^3\times I)$ which is the identity on the boundary. Fixing an element of~$\Diff(D^4, \partial D^4)$ we can further extend to a diffeomorphism of~$\widehat{M}$.
Using the Smale conjecture~\cite{Hatcher}, one can show that the isotopy class is independent of the choice made when extending to ${M}\cup(S^3\times I)$. Therefore for each fixed class in~$\pi_0(\Diff(D^4, \partial D^4))$ we obtain a homomorphism
\[\pi_1(\uemb(L,S^3))\to\pi_0(\Diff(\widehat{M})).\]
It would be interesting to compute when these classes are non-trivial, using Corollary~\ref{abccor - motion group S3} to understand the left hand side. Note that the image of this map consists of diffeomorphisms which act trivially on second homology, and thus are topologically isotopic to the identity \cite{Quinn}. 

To give an example, when $L$ is a disjoint union of $n$ unknots with framing $\pm1$ and $m$ Hopf links with zero-framing, $\widehat{M}$ is a connected sum of $n$ $\C P^2$ or $\overline{\C P}^2$'s and $m$ $S^2\times S^2$'s. In this case, we can compute the fundamental group of $\uemb(L,S^3)$ explicitly using Corollary \ref{abccor - motion group S3} (see Example \ref{example - Htrivial} for a description of $\pi_1(\link(L))$).

One can similarly construct maps 
$\pi_k(\uemb(L,S^3))\to\pi_{k-1}(\Diff(\widehat{M})),$
 for each fixed class in $\pi_{k-1}(\Diff(D^4, \partial D^4))$, but in this setting we do not have a version of Corollary~\ref{abccor - motion group S3}.

\subsection{Outline}
In Section~\ref{section - separating systems} we define the spaces of interest, introduce separating systems, and prove some combinatorial and topological properties which we need for the computation of the motion group. In Section~\ref{section - semi simplicial} we introduce the semi-simplicial space and using semi-simplicial arguments we reduce the proof of Theorem~\ref{abcthm - sep equiv} to a contractibility statement, the proof of which we defer to Section~\ref{section - contractible}.
We study~$\pi_1(\link(L))$ in Section~\ref{section - the fundamental group}. We prove Theorem~\ref{abcthm - motion group} and extend our result to the motion group of link configurations in~$S^3$ as opposed to~$\R^3$, proving Corollary~\ref{abccor - motion group S3}. 
Following this in Section~\ref{section - contractible} we prove the contractibility statement using geometric tools. Thus we have completed the proof of Theorem~\ref{abcthm - sep equiv}. Finally, in Appendix~\ref{section - implications for homotopy groups} we discuss some further implications of this result for finding representatives of homotopy classes of~$\link(L)$. 

\subsection{Acknowledgements}
We are very grateful to Gabriel Corrigan, Kiyoshi Igusa, Oscar Randal- Williams, and Jan Steinebrunner for insightful conversations. We would also like to thank Mark Powell and Peter Teichner for conversations about 4-manifold applications. Additionally we would like to thank the anonymous referees for helpful feedback, which significantly improved the paper.

The first author was partially supported by the Max Planck Institute for Mathematics in Bonn, ERC grant No.~756444, and EPSRC Fellowship No.~EP/V043323/1. 
The second author was supported by NSF grant DMS-2052801. 
Both authors acknowledge support and hospitality of the ``Research in Pairs'' program at the Mathematisches Forschungsinstitut Oberwolfach in August 2021, where some of this work was carried out.

\section{Embedding spaces and sphere systems}\label{section - separating systems}
In this section, we set up some notation and define the basic objects of study that will be referred to throughout the paper. We then prove some combinatorial and topological properties required for the proof of Theorem~\ref{abcthm - motion group}.

\subsection{Embedding spaces} Let $M$ be a smooth manifold. The diffeomorphism group of~$M$ will be denoted $\Diff(M)$, and the space of smooth embeddings of a smooth manifold $N$ into $M$, endowed with the $C^\infty$ Whitney topology, will be denoted $\emb(N,M)$. Recall that equipped with this topology $\emb(N,M)$ is a principal $\Diff(N)$-bundle where $\Diff(N)$ acts by precomposition.  The quotient is the space of unparametrised embeddings \[\uemb(N,M)=\emb(N,M)/\Diff(N).\]
We will typically work with $\uemb$ and thus often identify embeddings of spaces with their images.

A link $L$ is an element of $ \uemb(\sqcup_m S^1,\R^3)$ for some $m$.  We denote the path component of $ \uemb(\sqcup_m S^1,\R^3)$ containing $L$ by $\link(L)$. $L$ is called \emph{split} if it can be written as a disjoint union $L=L_1\sqcup\cdots\sqcup L_n$ where $n\geq 2$ and each $L_i\neq \emptyset$ can be contained in a ball $B_i$ such that $B_i\cap B_j=\emptyset$ for $i\neq j$. If no such decomposition exists $L$ is called \emph{unsplit}. We will always assume that each $L_i$ is unsplit in the decomposition $L=L_1\sqcup\cdots\sqcup L_n$. As in the introduction, the unsplit sublinks $L_i$ will be called the \emph{pieces} of $L$.

We now introduce the main spaces of study. For~$i\in \{1,2,\ldots,n\}$, fix once and for all unsplit links~$L_i\in \uemb(\sqcup_{m_i}S^1, \R^3)$
for some~$m_i\in \N$. We choose $L_i$ so that its image lies in the ball $B_i$ of radius $\frac{1}{2}$ centered on $(2i,0,0)\in \R^3$. We also require that if~$L_i$ and~$L_j$ are isotopic, then the embeddings differ by a translation of $\R^3$. Then for~$m=\sum_{i=1}^n m_i$ we set~$L$ to be the image of the disjoint union $L_1\sqcup\cdots \sqcup L_n$ in
$\uemb(\sqcup_{m}S^1, \R^3)$. When we refer to $L\in \link(L)$ in the sequel, we mean this chosen unparametrised embedding, which will serve as the natural basepoint for $\link(L)$.

We also introduce the space $\plink(L)$ to be the component of the embedding space
$$
\emb(\sqcup_{m}S^1,\R^3)/\prod_{i=1}^n\Diff(\sqcup_{m_i}S^1)
$$
containing~$L_1\sqcup\cdots \sqcup L_n$. This is the pure unparametrised embedding space, in analogy with the ``pure" configuration space, where each piece~$L_i$ is unparametrised but labelled. 
We use the notation~$L$ for both basepoints, as it is always clear as to which space we are working in.

As discussed at the end of Section~\ref{section - semi simplicial},~$\plink(L)$ is a finite cover of~$\link(L)$.

\subsection{Separating Systems}

\begin{defn}
A \emph{sphere system} for an embedding $\rho\in \link(L)$ or $\plink(L)$ is an embedding of a disjoint union of spheres $\Sigma\colon\sqcup_{i=1}^kS^2\hookrightarrow \R^3\setminus \rho$ in~$\uemb(\sqcup_{i=1}^kS^2,\R^3).$
\end{defn}

\begin{rem}
Note that although the pieces are labelled in~$\plink(L)$, we do not require the spheres in a separating system for~$\rho \in \plink(L)$ to be labelled.
\end{rem}

\begin{defn}\label{def-Separating System}
A \emph{separating system} $\Sigma$ for $\rho\in \link(L)$ or $\plink(L)$ is a sphere system such that each component of $\R^3\setminus \Sigma$ contains at most one piece of $\rho$. We call $\Sigma$ \emph{essential} if no connected component of $\R^3\setminus (\rho\cup \Sigma)$ is homeomorphic to $\Int(\B^3)$ nor $\Int(S^2\times I)$.
\end{defn}

The requirements for an essential separating system in Definition \ref{def-Separating System} are equivalent to the following:
\begin{itemize}
\item every sphere bounds a ball with at least one piece~$L_i$ inside it;
\item no two spheres are isotopic;
\item there is no sphere which bounds a ball containing the whole link.
\end{itemize}

For a link with~$n$ pieces, there must be at least~$n-1$ spheres in an essential separating system, but in general there may be more than~$n$, subject to the rules above. Examples of separating systems for a link with 3 pieces are shown in Figure \ref{fig:SeparatingExample}.

\begin{figure}[h]
\begin{tikzpicture}[framed, scale=0.6]
 \shade[ball color = cyan!40, opacity = 0.2] (0,0) circle (1.5cm);
 \shade[ball color = cyan!40, opacity = 0.4] (-.4,-.35) circle (0.35cm);
 \shade[ball color = cyan!40, opacity = 0.4] (2.5,2) circle (0.75cm);
 \draw (0.1,.2) arc (-20:320:.3cm and .2cm);
 \draw (0,.35)  arc (-160:180:.3cm and .4cm);
 \draw (2.35,2.12) arc (60:390:.2cm and .1cm);
 \draw[rotate=18] (2.9,1.1)  arc (-160:180:.3cm and .2cm);
 \draw[rotate=40] (-0.65,-.05) circle [x radius=.1cm, y radius=.2cm];
 \shade[ball color = pink!40, opacity = 0] (1.5,.3) circle (2.75cm);
 
\end{tikzpicture}
\qquad
\begin{tikzpicture}[framed, scale=0.6]
 \shade[ball color = amber!40, opacity = 0.2] (0,0) circle (1.5cm);
 \shade[ball color = amber!40, opacity = 0.4] (.25,.3) circle (0.8cm);
 \shade[ball color = amber!40, opacity = 0.4] (2.5,2) circle (0.75cm);
 \shade[ball color = pink!40, opacity = 0] (1.5,.3) circle (2.75cm);
 \draw (0.1,.2) arc (-20:320:.3cm and .2cm);
 \draw (0,.35)  arc (-160:180:.3cm and .4cm);
 \draw (2.35,2.12) arc (60:390:.2cm and .1cm);
 \draw[rotate=18] (2.9,1.1)  arc (-160:180:.3cm and .2cm);
 \draw[rotate=40] (-0.65,-.05) circle [x radius=.1cm, y radius=.2cm];

\end{tikzpicture}
\\ \vspace{1em}
\begin{tikzpicture}[framed, scale=0.6]
 \shade[ball color = cyan!40, opacity = 0] (0,0) circle (1.5cm);
 \shade[ball color = burlywood!40, opacity = 0.4] (-.6,-.5) circle (0.35cm);
 \shade[ball color = burlywood!40, opacity = 0.4] (.25,.45) circle (0.8cm);
 \shade[ball color = burlywood!40, opacity = 0.4] (2.5,2) circle (0.75cm);
 \shade[ball color = pink!40, opacity = 0] (1.5,.3) circle (2.75cm);
 \draw (0.1,.2) arc (-20:320:.3cm and .2cm);
 \draw (0,.35)  arc (-160:180:.3cm and .4cm);
 \draw (2.35,2.12) arc (60:390:.2cm and .1cm);
 \draw[rotate=18] (2.9,1.1)  arc (-160:180:.3cm and .2cm);
 \draw[rotate=40] (-0.65,-.05) circle [x radius=.1cm, y radius=.2cm];
 
\end{tikzpicture}
\qquad
\begin{tikzpicture}[framed, scale=0.6]
 \shade[ball color = cyan!40, opacity = 0] (0,0) circle (1.5cm);
 \shade[ball color = pink!60, opacity = 0.4] (.27,0) circle (1.2cm);
 \shade[ball color = pink!40, opacity = 0.4] (-.4,-.35) circle (0.35cm);
 \shade[ball color = pink!40, opacity = 0.2] (1.5,.3) circle (2.75cm);
 \draw (0.1,.2) arc (-20:320:.3cm and .2cm);
 \draw (0,.35)  arc (-160:180:.3cm and .4cm);
 \draw (2.35,2.12) arc (60:390:.2cm and .1cm);
 \draw[rotate=18] (2.9,1.1)  arc (-160:180:.3cm and .2cm);
 \draw[rotate=40] (-0.65,-.05) circle [x radius=.1cm, y radius=.2cm];

\end{tikzpicture}

 \caption{Separating systems for a link $\rho \in \link(L)$ with three pieces.  Only the bottom right separating system is not essential, because the unbounded component is homeomorphic to $\Int(S^2\times I)$.}
 \label{fig:SeparatingExample}
\end{figure}
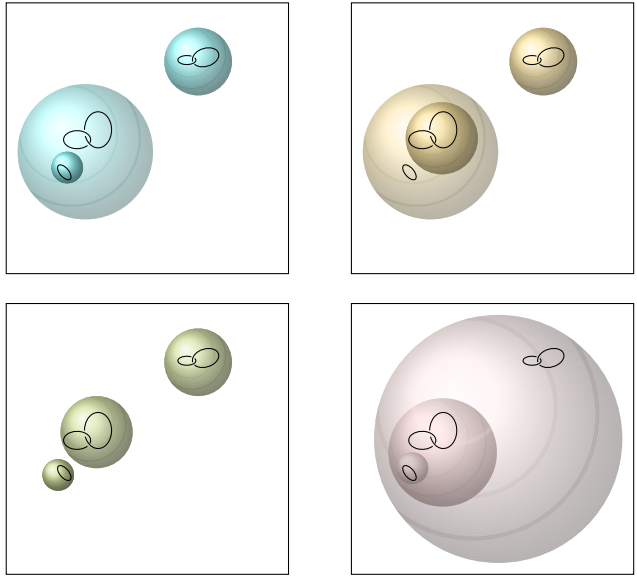

\begin{exam}\label{example - basepoint SS}
For our chosen basepoint $L$ in either~$\link(L)$ or~$\plink(L)$, recall that the pieces sit inside the balls $B_i$ of radius $\frac{1}{2}$ centered on $(2i,0,0)\in \R^3$, for~$1\leq i \leq n$. Then taking the unparametrised embedding which has image the boundary spheres of the~$B_i$ gives a canonical separating system~$\S_L$ for~$L$.
\end{exam}

We now consider embeddings in~$\link(L)$ together with a separating system $\Sigma$. For~$k$ a positive integer, consider the subspace
\[
\ESS(L,k)\subseteq\uemb(\sqcup_m S^1\sqcup_k S^2, \R^3)
\]
such that the image~$\rho$ of~$\sqcup_m S^1$ is isotopic to~$L$, and the image of~$\sqcup_kS^2$ gives an essential separating system for $\rho$ with $k$ spheres \emph{i.e.,}
\[\ESS(L,k)=\{(\rho,\Sigma)\mid \rho\in \link(L),~ \text{$\Sigma$ essential, separating system for~$\rho$ and } |\Sigma|=k\}.\]
Let~$\ESS(L)$ be the space
\[
\ESS(L)=\bigsqcup_{k=1}^\infty \ESS(L,k)
\]
of {essential separating systems} for links isotopic to~$L$, \emph{i.e.,} 
\[\ESS(L)=\{(\rho,\Sigma)\mid \rho\in \link(L),~ \text{$\Sigma$ essential separating system for~$\rho$}\}.\]
We can similarly define a space of labelled links and separating systems:
\[\PESS(L)=\{(\rho,\Sigma)\mid \rho\in \plink(L),~ \text{$\Sigma$ essential separating system for~$\rho$}\}.\]

The point~$(L,\S_L)$ is a natural basepoint for~$\ESS(L)$ and $\PESS(L)$.

\subsection{Combinatorics of separating systems}
We associate a rooted tree to each essential separating system, which is labelled when the link is.  Although  there are only finitely many possible such rooted trees up to isomorphism, there are infinitely many isotopy classes of essential separating systems in $\R^3\setminus \rho$ for each $\rho\in \link(L)$. We will make use of these rooted trees in Section \ref{section - the fundamental group} when we compute the fundamental group of $\link(L)$.
\begin{defn}\label{def-combinatorial type} Recall that ~$\rho=\rho_1\sqcup \ldots\sqcup \rho_n \in \plink(L)$ has a labelling of its pieces by the set~$\{1,\ldots,n\}$. Without loss of generality,~$\rho_i$ is labelled by $i$ and satisfies~$\rho_i\simeq L_i$. 
The \emph{combinatorial type} of a separating system $\Sigma$ for $\rho$ is the labelled rooted graph $T_\Sigma)$ whose vertices are connected components of $\R^3\setminus \Sigma$, and where an edge connects two vertices if they are separated by a single sphere of $\Sigma$. Let the root of the graph be the unique unbounded component of the complement. The other vertices are labelled by $i$ if the unique labelled link piece contained in that component is $\rho_i$ (and hence labelled by~$i$) and by $\emptyset$ otherwise. 
\end{defn}
\begin{figure}[h]
\begin{tikzpicture}[ scale=0.6]
 \shade[ball color = burlywood!40, opacity = 0.2] (0,0) circle (1.5cm);
 \shade[ball color = burlywood!40, opacity = 0.4] (-.6,-.5) circle (0.35cm);
 \shade[ball color = burlywood!40, opacity = 0.4] (.25,.45) circle (0.8cm);
 \shade[ball color = burlywood!40, opacity = 0.4] (2.5,2) circle (0.75cm);
 \shade[ball color = pink!40, opacity = 0] (1.5,.3) circle (2.75cm);
 \draw (0.1,.2) arc (-20:320:.3cm and .2cm);
 \draw (0,.35)  arc (-160:180:.3cm and .4cm);
 \draw (2.35,2.12) arc (60:390:.2cm and .1cm);
 \draw[rotate=18] (2.9,1.1)  arc (-160:180:.3cm and .2cm);
 \draw[rotate=40] (-0.65,-.05) circle [x radius=.1cm, y radius=.2cm];
 \draw[thick](-1,3)--(-.85,.4);
 \draw[thick](-1,3)--(2.5,2);
 \draw[thick](-.85,.4)--(-.6,-.5);
 \draw[thick](-.85,.4)--(.25,.45);
 \filldraw (-1,3) circle(2pt);
 \draw (-1, 3) circle(4pt);
 \filldraw (-.85,.4) circle(2pt);
 \filldraw (2.5,2) circle(2pt);
 \filldraw (-.6,-.5) circle(2pt);
 \filldraw (.25,.45) circle(2pt);
\node at (3.1,2) {$\scriptstyle 3$};
 \node at (.05,-.7) {$\scriptstyle1$};
 \node at (.8,.45) {$\scriptstyle2$};

\end{tikzpicture}
\qquad
\begin{tikzpicture}[ scale=0.6]
 \shade[ball color = cyan!40, opacity = 0] (0,0) circle (1.5cm);
 \shade[ball color = burlywood!40, opacity = 0] (-.6,-.5) circle (0.35cm);
 \shade[ball color = burlywood!40, opacity = 0] (.25,.45) circle (0.8cm);
 \shade[ball color = burlywood!40, opacity = 0] (2.5,2) circle (0.75cm);
 \shade[ball color = pink!40, opacity = 0] (1.5,.3) circle (2.75cm);
 \draw[thick](-1,3)--(-.85,.4);
 \draw[thick](-1,3)--(2.5,2);
 \draw[thick](-.85,.4)--(-.6,-.5);
 \draw[thick](-.85,.4)--(.25,.45);
 \filldraw (-1,3) circle(2pt);
 \draw (-1, 3) circle(4pt);
 \filldraw (-.85,.4) circle(2pt);
 \filldraw (2.5,2) circle(2pt);
 \filldraw (-.6,-.5) circle(2pt);
 \filldraw (.25,.45) circle(2pt);
 \node at (-1.35,3) {$\scriptstyle \emptyset$};
 \node at (2.85,2) {$\scriptstyle3$};
 \node at (-.25,-.7) {$\scriptstyle1$};
 \node at (.65,.45) {$\scriptstyle2$};
 \node at (-1.1,.4) {$\scriptstyle \emptyset$};
 
\end{tikzpicture}
\caption{An essential separating system $\Sigma$ for~$\rho=\rho_1\sqcup \rho_2\sqcup \rho_3 \in \plink(L)$ and its dual rooted $L$-tree $T_\Sigma$.}
\label{fig:DualTree}
\end{figure}
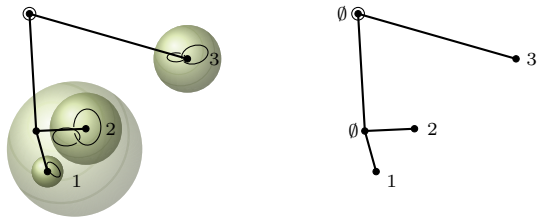
\begin{rem}\label{rem:CombLTree}
Since an embedded sphere separates $\R^3$ into exactly two components and the spheres of $\Sigma$ are disjoint, $T_\Sigma$ is always a tree. Observe that $\Sigma$ is essential if and only if there does not exist a univalent vertex labelled by $\emptyset$, nor a bivalent vertex, unless that vertex is the root.  
\end{rem}

\begin{rem}
One could similarly define the combinatorial type of a separating system for~$\rho \in \link(L)$, but some care must be taken with labels -- for example one could label by isotopy classes of links, since we do not differentiate between isotopic pieces. Since in Section \ref{section - the fundamental group} we only need the notion of combinatorial type for labelled links, we do not explore this further.
\end{rem}

\begin{defn}\label{def:Ltree}
For $n\geq 1$, a \emph{rooted $L$-tree} is a rooted tree whose vertices carry labels in the set $\{1,\ldots, n,\emptyset\}$ such that 
\begin{enumerate}
\item Labels $1,\ldots, n$ appear exactly once.  
\item There are no univalent or bivalent vertices labelled by $\emptyset$, unless it is a bivalent vertex which is also the root.

\end{enumerate}
\end{defn}
Any vertex can serve as the root of a rooted $L$-tree. The root is distinguished by a circled vertex as in Figure~\ref{fig:DualTree}.

\begin{lem}\label{lem:FinTree}
There are finitely many rooted $L$-trees.

\begin{proof}
Observe that for any rooted $L$-tree $T$, the number of vertices and edges will each increase by 1 under the following operations
\begin{itemize}
\item Making an interior $i$-vertex into a leaf, as in Figure~\ref{fig:L_i Leaf},
\item Splitting an interior $\emptyset$-vertex of valence at least 4 into two $\emptyset$-vertices of valence at least 3 with an edge between them, as in Figure~\ref{fig:Splitting empty}.
\item Splitting an $\emptyset$-labelled root vertex of valence at least 3 into two $\emptyset$-labelled vertices, one of which is the new root and has valence at least 2, while the other has valence at least 3, as in Figure~\ref{fig:Splitting root}. 
\end{itemize}
As a consequence, the maximal number of edges occurs when each $i$-vertex is a leaf, each interior non-root $\emptyset$-vertex has valence 3, and the root is labelled by $\emptyset$ and has valence 2. An Euler characteristic calculation shows that the maximal number of vertices in such a tree is $2n-1$. Hence, there are only finitely many such $T$.
\end{proof}
\end{lem}

\begin{figure}[h!]
     
     \begin{subfigure}[b]{\textwidth}
     \centering
        \begin{tikzpicture}[scale=0.6]
        \filldraw (0,0) circle(4pt);
        \draw[thick](0,0)--(-1.5,-1);
        \draw[thick](0,0)--(-.25,-1);
        \draw[thick](0,0)--(1,-1);
        \draw[thick](0,0)--(1,1.5);
        \node at (3,0) {$\Rightarrow$};
        \node at (-.25,.3){$\scriptstyle i$};
        \end{tikzpicture}
        \qquad
        \begin{tikzpicture}[scale=0.6]
        \filldraw (0,0) circle(4pt);
        \filldraw (-1,.75) circle(4pt);
        \draw[thick](0,0)--(-1.5,-1);
        \draw[thick](0,0)--(-.25,-1);
        \draw[thick](0,0)--(1,-1);
        \draw[thick](0,0)--(1,1.5);
        \draw[thick](0,0)--(-1,.75);
        \node at (-1.3,1.05){$\scriptstyle i$};
        \node at (.35,0){$\scriptstyle \emptyset$};
        \end{tikzpicture}
         \caption{Creating an $L_i$-labelled leaf.}
         \label{fig:L_i Leaf}
     \end{subfigure}
     \vspace{1em}
     
     \begin{subfigure}[b]{\textwidth}
        \centering
         \begin{tikzpicture}[scale=0.6]
        \filldraw (0,0) circle(4pt);
        \draw[thick](0,0)--(-1.2,1);
        \draw[thick](0,0)--(-1.3,.25);
        \draw[thick](0,0)--(-1.4,-.75);
        \draw[thick](0,0)--(1.2,.8);
        \draw[thick](0,0)--(1.2,-.4);
        \node at (0,.4){$\scriptstyle\emptyset$};
        \node at (3,0) {$\Rightarrow$};
        \end{tikzpicture}
        \qquad
        \begin{tikzpicture}[scale=0.6]
        \filldraw (0,0) circle(4pt);
        \filldraw (2,0) circle(4pt);
        \draw[thick](0,0)--(-1.2,1);
        \draw[thick](0,0)--(-1.3,.25);
        \draw[thick](0,0)--(-1.4,-.75);
        \draw[thick](2,0)--(3.2,.8);
        \draw[thick](2,0)--(3.2,-.4);
        \draw[thick](0,0)--(2,0);
        \node at (0,.4){$\scriptstyle\emptyset$};
        \node at (2,.4){$\scriptstyle\emptyset$};
        \end{tikzpicture}
         \caption{Splitting an $\emptyset$-labelled vertex.}
         \label{fig:Splitting empty}
     \end{subfigure}
         \vspace{1em}
         
     \begin{subfigure}[b]{\textwidth}
         \centering
            \begin{tikzpicture}[scale=0.6]
        \filldraw (0,0) circle(4pt);
        \draw (0,0) circle (6pt);
        \draw[thick](0,0)--(-1.2,1);
        \draw[thick](0,0)--(-1.4,-.75);
        \draw[thick](0,0)--(1.5,0);
        \node at (0,.5){$\scriptstyle\emptyset$};
        \node at (3,0) {$\Rightarrow$};
        \end{tikzpicture}
        \qquad
        \begin{tikzpicture}[scale=0.6]
        \filldraw (0,0) circle(4pt);
        \filldraw (2,0) circle(4pt);
        \draw (2,0) circle (6pt);
        \draw[thick](0,0)--(-1.2,1);
        \draw[thick](0,0)--(-1.4,-.75);
        \draw[thick](0,0)--(2,0);
        \draw[thick](2,0)--(3,0);
        \node at (0,.5){$\scriptstyle\emptyset$};
        \node at (2,.5){$\scriptstyle\emptyset$};
        
        \end{tikzpicture}
        
         \caption{Splitting an $\emptyset$-labelled root vertex.}
         \label{fig:Splitting root}
     \end{subfigure}
        \caption{Three moves on $L$-trees.}
        \label{fig:three moves}
\end{figure}
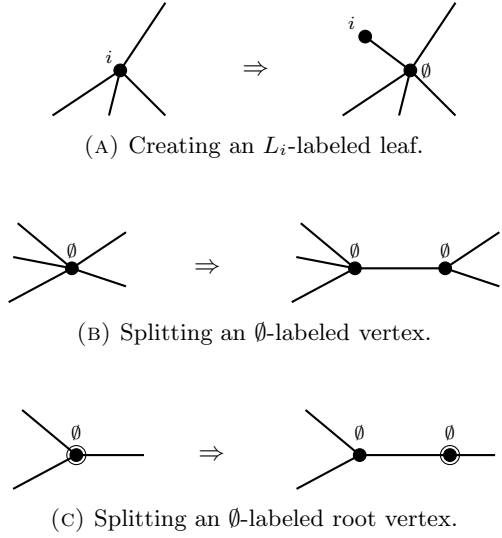

\begin{lem}\label{lem:ComType} For any $\rho\in \plink(L)$, there is a bijection between combinatorial types of essential separating systems for $\rho$ and rooted $L$-trees. 

\begin{proof}
Since any two embeddings $\rho, ~\rho'\in\plink(L)$ are ambient isotopic, there is a bijection between combinatorial types of essential separating systems for $\rho$ and $\rho'$.  Thus, it suffices to prove the lemma for the fixed embedding $L\in \plink(L)$. Recall that the image of $L_i$ lies in a ball of radius $1/2$ centered on the point $(2i,0,0)\in \R^3$.

By Remark \ref{rem:CombLTree}, for any essential separating system $\Sigma$, we know that $T_\Sigma$ is a rooted $L$-tree. Now let $T$ be any rooted $L$-tree. We claim that there exists an essential separating system $\Sigma$ for $L$ such that $T_\Sigma=T$.

The proof is by induction on the number of pieces of $L$.  If $L$ has one piece,  the only rooted $L$-tree is where $T$ is a single rooted vertex. In this case the only essential separating system is empty. Hence this realises $T$ and uniqueness up to isotopy is vacuous.  Now suppose that there are $n\geq 2$ pieces. Then $T$ has at least two vertices and since it is a tree, it has at least leaf  $w$ which is not the root. Reordering the link pieces if necessary, we may assume that $w$ is labelled by $n$. Let $v$ be the vertex adjacent to $w$. Deleting $w$ and its edge from $T$ yields a tree $T'$ labelled by $1,\ldots, {n-1}$ and $\emptyset$.   Condition (1) of Definition~\ref{def:Ltree} is still satisfied, as well as condition (2), unless 
\begin{enumerate}[(i)]
\item $v$ is not the root, is labelled by $\emptyset$, and has valence 3 in $T$, with adjacent vertices~$u_1, u_2$ and $w$. 
\item $v$ is the root, is labelled by $\emptyset$ and has valence 2 in $T$.
\end{enumerate}

If (i) occurs, we set $T'$ to be the tree where the star of $v$ has been replaced by a single edge between the vertices~$u_1$ and~$u_2$. If (ii) occurs,  we set $T'$ to be the tree where $v$ and~$w$ are deleted (along with the edges that have an end point at~$v$) and the other vertex $w'$ adjacent to $v$ is now the root.

By induction, we can find $\Sigma'$ realising $T'$ for $L_1\sqcup\cdots \sqcup L_{n-1}$. First, assume that we do not fall into either of the cases (i) or (ii) above.  Then $v$ is still a vertex in $T'$. Let $U$ be the component of $\R^3\setminus \Sigma'$ corresponding to $v$. Now choose a point $p\in U\setminus (L_1\sqcup\cdots \sqcup L_{n-1})$ and a small closed ball $B\subset U$ centered on $p$. Choose $\tau\in\link(L_n)$ such that the image of $\tau$ lies in the interior of $B$.  There is an isotopy simultaneously taking $B$ to the ball of radius $1/2$ centered on $2n$ and the image of $\tau$ onto $L_n$. Now apply isotopy extension to obtain an ambient isotopy of $\R^3$, from which we obtain a sphere system $\Sigma$ realising $T$, where~$w$ corresponds to the interior of~$B$, and we add the boundary of this ball to obtain~$\S$.

If we are in case (i),  we take two parallel copies of the sphere of $\Sigma'$ corresponding to the edge added between the vertices~$u_1$ and~$u_2$ and set $U$ to be the component in between these two parallel copies.  If we are in case (ii), then $w'$ is the root, hence corresponds to the unbounded component. Add a large sphere surrounding all of $\Sigma'$ and set $U$ to be the unbounded complement of this sphere. In each case, we then choose a ball $B\subset U$ and $\tau\in \link(L_n)$ whose image lies in $B$ and proceed as above. This completes the proof.
\end{proof}
\end{lem}

\subsection{Forgetting embedded balls} In this section we relate the homotopy type of $\link(L)$ to that of the space of embeddings of a link $L$ together with a disjoint collection of points or 3-balls. We will use the results of this section to obtain our description of the fundamental group of $\link(L)$ in Section \ref{section - the fundamental group}.

Given a manifold $M$ let $\pconf_k(M)$ denote the configuration space of $k$ \emph{ordered} (or labelled) points in~$M$. We will use the notation $\{p_1,\ldots, p_k\}$ for a point in this space, and  $\uemb(L\sqcup\{p_1,\ldots,p_k\},\R^3)$ for the space of unparametrised embeddings of $\sqcup_mS^1\sqcup\{p_1,\ldots,p_k\}$ where the $k$ points are ordered and the image of $\sqcup_mS^1$ is isotopic to $L$. 

If instead of points we consider a disjoint union of $k$ labelled 3-balls, we write $B_1\sqcup\cdots\sqcup B_k$ and obtain an embedding space $\uemb(L\sqcup B_1\sqcup\cdots \sqcup B_k,\R^3)$. In terms of parametrised embeddings, we can describe  $\uemb(L\sqcup B_1\sqcup\cdots\sqcup B_k,\R^3)$ as the component of \[\emb(\sqcup_m S^1\sqcup B_1\sqcup\cdots\sqcup B_k,\R^3)/\left(\Diff(\sqcup_mS^1)\times \Diff(B_1)\times\cdots\times \Diff(B_k)\right)\]
where the image of $\sqcup_mS^1$ is isotopic to $L$.  That is, we do not quotient by the elements of $\Diff(\sqcup_k B^3)$ that permute distinct balls.

\begin{lem}\label{lem:equiv between balls and config}The map which takes centerpoints yields the equivalence
\[ \uemb(B_1\sqcup\cdots\sqcup B_k,\R^3)\simeq \pconf_k(\R^3)
\]
\end{lem}
\begin{proof}
    By linearisation, there is an equivalence
\[
\emb(B_1\sqcup\cdots\sqcup B_k,\R^3)\simeq \pconf^{\rm framed}_k(\R^3)
\]
where for each embedded ball the framed point in $\R^3 \times \GL(3)$ is given by taking the image and derivative of the embedding at $0\in B^3$. Since $\Diff(B_1)\times\cdots\times \Diff(B_k)\simeq \GL(3)^k$ taking the quotient by this group yields the result.
\end{proof}

 \begin{lem}\label{lem:Balls Complement} Let $L$ be any (possibly empty) link.  For any $k\geq 1$ there exists a homotopy equivalence
 \begin{equation}
   \uemb(L\sqcup B_1\sqcup\cdots\sqcup B_k,\R^3)\xrightarrow[]{\simeq}\uemb(L\sqcup\{p_1,\ldots,p_k\},\R^3)
 \end{equation}
 \end{lem}
 \begin{proof} 
Consider the Palais fibrations (see, e.g.~\cite[Corollary 2.5]{BoydBregmanSteinebrunner24})
\[ \uemb(L \sqcup B_1\sqcup\cdots\sqcup B_k,\R^3)\rightarrow \uemb(B_1\sqcup\cdots\sqcup B_k,\R^3)
 \]
 and 
 \[ \uemb(L \sqcup\{p_1,\ldots,p_k\},\R^3)\rightarrow \uemb(\{p_1,\ldots,p_k\},\R^3)=\pconf_k(\R^3).
 \]
 By Lemma \ref{lem:equiv between balls and config}, the bases are homotopy equivalent, and the fibers are equivalent, since $\R^3\setminus (B_1\sqcup\cdots\sqcup B_k)$ is homeomorphic to $\R^3\setminus \{p_1,\ldots, p_k\}$.
\end{proof}

\section{Semi-simplicial spaces of separating systems}\label{section - semi simplicial}
This section introduces the main semi-simplicial space we work with, and contains the semi-simplicial arguments we require to prove Theorem~\ref{abcthm - sep equiv}. 

Recall that \[\ESS(L)=\{(\rho,\Sigma)\mid \rho\in \link(L),~ \text{$\Sigma$ essential separating system for~$\rho$}\},\] 
topologised as a subspace of the product of unparametrised embedding spaces, with the~$C^\infty$ Whitney topology, and $\PESS(L)$ is the analogously defined space for  links in $\plink(L)$.

\begin{defn}\label{defn - sep and psep}
The semi-simplicial space $\Sep_\bt$ is defined as follows:
\begin{itemize}
    \item $\Sep_0=\ESS(L)$.
    \item $\Sep_p\subseteq (\Sep_0)^{p+1}$ consists of ordered~$(p+1)$ tuples~$(\rho_i,\Sigma_i)$ such that~$\rho_i=\rho_j$ for all~$0\leq i,j\leq p$ and~$\Sigma_i \cap \Sigma_j=\emptyset$ for all~$i\neq j$.  \item Face maps $\p_p^{i}\colon \Sep_p\rightarrow \Sep_{p-1}(L)$ for $0\leq i\leq p$ are given by forgetting the~$i$th entry in a tuple.
\end{itemize}
The semi-simplicial space $\pSep_\bt$ is defined analogously, replacing~$\ESS(L)$ with $\PESS(L)$.
\end{defn}

There is a natural augmentation~$\varepsilon$ to the embedding space~$\link(L)$ (which can be thought of as $\pSep_{-1}$) given by forgetting the sphere systems, and this is indeed an augmentation since it commutes with face maps. The  fiber of this augmentation for a fixed~$\rho\in \link(L)$ is itself a semi-simplicial space which levelwise is given by $\varepsilon^{-1}_p(\rho)$, and has face maps inherited from~$\Sep_\bt$. We denote this semi-simplicial space by $\Sep(\rho)_\bt$.  

\begin{lem}For any two embeddings $\rho, \rho' \in \link(L)$, $\Sep(\rho)_\bt\cong \Sep(\rho')_\bt$.

\begin{proof}
There exists an ambient isotopy $\Phi$ of $\R^3$ sending $\rho$ to $\rho'$.  Then $\Phi$ induces a bijective correspondence between separating systems for $\rho$ and separating systems for $\rho'$.  This correspondence clearly takes essential separating systems to essential separating systems, preserving disjointness and combinatorial type. Thus $\Phi$ induces an isomorphism $\Phi_\bt\colon \Sep(\rho)_\bt\rightarrow \Sep(\rho')_\bt$.
\end{proof}
\end{lem}

\begin{lem} \label{lem - homotopy fiber is amalg of levelwise fibers}
The map $|\Sep(\rho)_\bt|\to hofib_\rho|\varepsilon_\bt|$ is a weak equivalence.
\begin{proof}
By isotopy extension, one can lift an open neighbourhood around any point~$\rho$ in~$\link(L)$ to~$\Sep_p$, and it follows that the augmentation is a level-wise Serre fibration (compare with Section 4.4 in \cite{Kupers}). Now Lemma~2.14 in~\cite{EbertRandalWilliams} shows that since the maps are levelwise quasifibrations then for each~$\rho \in \link(L)$ the natural map
\[
|\Sep(\rho)_\bt|=|\varepsilon^{-1}_\bt(\rho)| \to hofib_\rho|\varepsilon_\bt|
\]
is a weak equivalence. 
\end{proof}
\end{lem}

\begin{defn}
We define some variations of~$\Sep(\rho)_\bt$. Let $\Sep(\rho)^\delta_\bt$ be the semi-simplicial space where~$\Sep(\rho)^\delta_p$ is equal to $\Sep(\rho)_p$ as a set, but is given the discrete topology. Fix a separating system~$\hat{\Sigma}$ for~$\rho$ (not necessarily essential), and let $\Sep(\rho,\hat{\Sigma})^\delta_p\subset \Sep(\rho)^\delta_p$ be the subspace such that the tuples of separating systems
$(\Sigma_0,\ldots,\Sigma_p)$ for~$\rho$ are disjoint from~$\hat{\Sigma}$. Let~$\Sep(\rho,\hat{\Sigma})^\delta_\bt$ be the semi-simplicial space with these levels and face maps induced by those of $\Sep(\rho)^\delta_\bt$.
\end{defn}

For the next proposition, we need the following lemma. We state it here, but reserve the proof until Section \ref{section - contractible}, when we have set up the appropriate notation. 

\begin{lem}\label{lem - discrete missing sep is contractible}
$|\Sep(\rho, \hat{\Sigma})^\delta_\bt|$ is contractible.
\end{lem}

\begin{prop}\label{prop - moving from discrete topology to whitney}
If $|\Sep(\rho)^\delta_\bt|$ is~$k$-connected, then so is $|\Sep(\rho)_\bt|$.
\begin{proof}
In this proof we use single bars for geometric realisation of a semi-simplicial set, and double bars for the iterated geometric realisation of a bi-semi-simplicial set.

Let~$\iota$ be the map $|\Sep(\rho)^\delta_\bt| \to |\Sep(\rho)_\bt|$ induced by the identity map.
We follow \cite{GalatiusRandalWilliams}. Consider the bi-semi-simplicial space~$D_{\bt,\bt}$ given by 
\[
D_{p,q}= \Sep(\rho)_{(p+q+1)}
\]
topologised as a subspace of~$\Sep(\rho)_p \times \Sep(\rho)^\delta_q$, and with face maps inherited from $\Sep(\rho)_\bt$. That is,~$D_{p,q}$ is the subspace of~$\Sep(\rho)_p \times \Sep(\rho)^\delta_q$ such that all~$(p+q+2)$ separating systems can be realised disjointly. Then there are two augmentation maps~$\epsilon$ and~$\delta$:
\[
\epsilon_p: |D_{p,\bt}| \to \Sep(\rho)_p \text{   and   } \delta_q: |D_{\bt,q}| \to \Sep(\rho)^\delta_q
\]
(\emph{i.e.,}~set~$D_{p,-1}=\Sep(\rho)_p$ and~$D_{-1,q}=\Sep(\rho)^\delta_q$). Then by \cite[Lemma 5.8]{GalatiusRandalWilliams} the following diagram commutes up to homotopy. \[
\xymatrix{& {\ar[ld]_{|\epsilon_\bt|}} {\ar@{}[d]^{}} \|D_{\bt,\bt}\|\ar[rd]^{|\delta_\bt|} & \\
|\Sep(\rho)_\bt| & & {\ar[ll]^{|\iota_\bt|}}|\Sep(\rho)^\delta_\bt|
}
\]
Now by \cite[Lemma 2.8]{GalatiusRandalWilliams}, $\epsilon_p$ is a Serre microfibration, since~$\Sep(\rho)_p$ is Hausdorff, and for a fixed separating system $(\S_0,\ldots, \S_p)\in \Sep(\rho)_p$, the subspace $\{x\in\Sep(\rho)_{q} |(\S_0,\ldots, \S_p, x) \in \Sep(\rho)_{p+q+1}\}$ is open in $\Sep(\rho)_{q}$.
Let~$(\S_0,\ldots, \S_p)\in \Sep(\rho)_p$, and $\hat{\Sigma}$ be the image of the separating system~$\S_0\sqcup \ldots \sqcup \S_p \in \R^3$. Then $\epsilon_p$ has fiber~$\Sep(\rho, \hat{\Sigma})^\delta_\bt$, and the realisation of this space is contractible by Lemma~\ref{lem - discrete missing sep is contractible}. It follows that~$|\epsilon_\bt|$ is a weak equivalence. Up to homotopy this map factors through $|\Sep(\rho)^\delta_\bt|$ (since $|\iota_\bt|\circ |\delta_\bt|\simeq |\epsilon_\bt|$) and so we conclude that if  $|\Sep(\rho)^\delta_\bt|$ is~$k$-connected, then so is $|\Sep(\rho)_\bt|$.
\end{proof}
\end{prop}

\begin{thm}\label{thm-discrete fiber is contractible}
$|\Sep(\rho)^\delta_\bt|$ is contractible.
\end{thm}

The proof of this theorem is the main focus of Section~\ref{section - contractible}. Assuming this result, we can now prove Theorem~\ref{abcthm - sep equiv}.

\begin{rem}\label{rem - ht cw complexes}
The spaces $\link(L)$, $\plink(L)$, $\Sep_\bt$, $\pSep_\bt$, and $\Sep(\rho)_\bt$ all have the homotopy type of CW-complexes by \cite{Palais, Milnor, Kuratowski}.  Kuratowski~\cite{Kuratowski} showed if $Y$ is an ANR and $X$ is compact then the space of continuous maps from $X$ to $Y$ is an ANR. Milnor~\cite{Milnor} showed that ANR spaces are exactly those with the homotopy type of CW-complexes. Finally, Palais~\cite{Palais} collected results on infinite dimensional manifolds, which included the space of smooth maps between manifolds. In particular they are ANRs, as are embedding spaces (since open subsets of ANRs are ANRs).
\end{rem}

\begin{restate}{Theorem}{abcthm - sep equiv}
The map $|\varepsilon_\bt|\colon|\Sep_\bt|{\rightarrow} \link(L)$
is a homotopy equivalence. 
\end{restate}

\begin{proof}
We first show the map is a weak equivalence  \emph{i.e.,}~for all~$k$ it induces isomorphisms
\[\pi_k(|\Sep_\bt|)\cong \pi_k(\link(L)).\] 
By Lemma~\ref{lem - homotopy fiber is amalg of levelwise fibers}, $hofib_\rho|\varepsilon_\bt|$ is weakly equivalent to $|\Sep(\rho)_\bt|$, 
and combining Proposition~\ref{prop - moving from discrete topology to whitney} and Theorem~\ref{thm-discrete fiber is contractible}, we conclude that the latter is contractible. It follows that~$hofib_\rho|\varepsilon_\bt|$ is weakly equivalent to a point and thus~$|\varepsilon_\bt|$ is a weak equivalence. By Remark~\ref{rem - ht cw complexes} the spaces both have the homotopy type of CW complexes and so~$|\varepsilon_\bt |$ is a homotopy equivalence.
\end{proof}

Recall from Definition~\ref{defn - sep and psep} that $\pSep_\bt$ is the semi-simplicial space constructed analogously to $\Sep_\bt$, but in which all link pieces are labelled. Then $\pSep_\bt$ has an augmentation to $\plink(L)$ that forgets the separating systems, which we also denote by $\varepsilon_\bt\colon \pSep_\bt\rightarrow \plink(L)$.  Consequently, there is a levelwise covering map $|\pSep_\bt|\rightarrow |\Sep_\bt|$ and a commutative pullback square
\begin{equation}\label{eq - pullback square}
\xymatrix{|\pSep_\bt|\ar[r]\ar[d]^{|\varepsilon_\bt|}&|\Sep_\bt|\ar[d]^{|\varepsilon_\bt|}\\
\plink(L)\ar[r]&\link(L)
}
\end{equation}
where the horizontal arrows are the covering maps induced by forgetting the labels on pieces. By Theorem \ref{abcthm - sep equiv} the right hand map is a homotopy equivalence.

\begin{cor}\label{cor - PSep Homotopy equivalence}
The map $|\varepsilon_\bt|\colon|\pSep_\bt|\rightarrow \plink(L)$ is a homotopy equivalence.
\end{cor}

\begin{proof}
Since (\ref{eq - pullback square}) is a pullback square, and the homotopy fibers of the right hand map are contractible it follows that the homotopy fibers of the lefts hand map are contractible and so it is a weak equivalence. By Remark~\ref{rem - ht cw complexes} both spaces have the homotopy type of CW complexes, and so the statement of the corollary follows.  
\end{proof}

\section{The fundamental group of~\texorpdfstring{$\link(L)$}{the link embedding space}}\label{section - the fundamental group}

In this section we use the previous results to give a general presentation for the fundamental group of~$\link(L)$, where $L$ is the image in~$\link(L)$ of the disjoint union $L_1\sqcup \cdots \sqcup L_n$. Recall that the $L_i$ satisfy the following criteria:
\begin{enumerate}[(i)]
    \item Each piece $L_i$ is contained in the ball $B_i$ of radius $\frac{1}{2}$ centered on $(2i,0,0)\in \R^3$.
    \item If $L_i\cong L_j$ for $i<j$ then $L_j=L_i+(2(j-i),0,0)$. That is, if two pieces are isotopic, then the unparametrised embeddings differ by a translation of $\R^3$.  
\end{enumerate}

The embedding $L$ will always serve as our chosen basepoint for $\pi_1$, hence in what follows we will write $\pi_1(\link(L))$ in place of $\pi_1(\link(L),L)$, and similarly for~$\plink(L)$. As a first step we identify the possible permutations of the pieces of $L$, and pass to studying~$\pi_1(\plink(L))$. The presentation of $\pi_1(\plink(L))$ will then depend on the following two families of groups:

\begin{itemize}
    \item $G_i=\pi_1(\link(L_i))$, and
    \item $H_i=\pi_1(\R^3\setminus L_i)$, for~$L_i$ a piece of~$L$.
\end{itemize}

We will describe $\pi_1(\plink(L))$ in terms of the $G_i$ and certain automorphisms of the fundamental group of the complement  $\pi_1(\R^3\setminus L)=H_L\cong H_1*\cdots*H_n$. Our strategy will be to identify certain subgroups of the motion group, and then show that any element of the motion group is a product of elements of these.

\subsection{Permutations of the pieces} The first subgroup we consider will be permutations of the link pieces. Let $S_{\{L_1,\ldots,L_n\}}$ denote the symmetric group on the set of link pieces $\{L_1,\ldots, L_n\}$ and  define $P_L$ to be the group of permutations  satisfying $L_i\cong L_{\sigma(i)}$ for all $\sigma\in P_L$.  That is, we only allow permutations in $P_L$ when link pieces are isotopic.  

\begin{lem}\label{lem:Permute Links}There is a surjective homomorphism
$\eta\colon\pi_1(\link(L))\rightarrow P_L$.
\end{lem}
\begin{proof}  An element of $\pi_1(\link(L))$ is represented by a path of unparametrised embeddings $\gamma=\gamma_t\colon[0,1] \rightarrow \link(L)$ such that $\im(\gamma_0)=\im(\gamma_1)=L$.  Thus, $\gamma$ induces a permutation of the pieces of $L$, and we define $\eta(\gamma)\in S_{\{L_1,\ldots,L_n\}}$ to be this permutation.  Since $\gamma$ can only exchange $L_i$ and $L_j$ when $L_i\cong L_j$, the image of $\eta$ lies in $P_L$. To see that $\eta$ is surjective, we use criterion (ii) in the definition of $L$ to produce lifts of elements in~$P_L$, \emph{i.e.}~a sequence of translations of the $B_i$ which exchange $L_i$ with $L_j$ when $L_i\cong L_j$. Explicitly, for $L_i\cong L_j,~i<j$: translate $B_i$ up to $B_i+(0,0,2)$ then over to $B_i+(2(j-i),0,2)$ then down to $B_i+(2(j-i),0,0)$, while simultaneously translating $B_j$ down to $B_j+(0,0,-2)$ then to $B_j+(2(i-j),0,-2)$ then up to $B_j+(2(i-j),0,0)$.  
\end{proof}

It follows from Lemma \ref{lem:Permute Links} that after passing to a finite cover of $\link(L)$ of degree $|P_L|$, we may assume that any loop of embeddings returns each piece of $L$ back to itself. Otherwise stated, we can regard this cover as embeddings in $\link(L)$ where the pieces are labelled. This is precisely the space $\plink(L)$ with labelled basepoint~$L$. In order to simplify exposition, we will now assume that the pieces of $L$ are labelled, and revisit the unlabelled case at the end.

\subsection{Motions of individual pieces} The second subgroup we will identify consists of motions of the individual pieces $L_i$. Let $\uemb(L_i,B_i)$ be the path component of the space of unparametrised embeddings of $\sqcup_{m_i} S^1$ into the ball $B_i$ containing $L_i$. By identifying the interior of $B_i$ with $\R^3$, we obtain a homeomorphism $\link(L_i)\cong \uemb(L_i,B_i)$. Thus, the inclusion $B_i\rightarrow \R^3$ induces a weak equivalence $\uemb(L_i,B_i)\rightarrow \link(L_i)$, where the inverse map is given by shrinking $\R^3$ into $B_i$. Given any element of $ \uemb(L_i,B_i)$ we obtain an element of $\plink(L)$ by mapping to the basepoint $L_j$ on the pieces $L_j\neq L_i$. This defines an embedding $\beta_i\colon\link(L_i)\hookrightarrow \plink(L)$.  

\begin{lem}\label{lem:Snow Globe}Let $(\beta_i)_*$ be the induced map on $\pi_1$.  Then $(\beta_i)_*$ is injective.
\begin{proof}
As discussed above, since $B_i\cong \R^3$, the inclusion $B_i\hookrightarrow \R^3$ induces an isomorphism from $\pi_1(\uemb(L_i,B_i))\cong \pi_1(\link(L_i))$. Now we have a composition \[\uemb(L_i,B_i)\hookrightarrow \plink(L)\rightarrow \link(L_i)\]
where the first map is $\beta_i$ and the second map forgets all pieces except $L_i$. Since the composition is the inclusion from above, $(\beta_i)_*$ must be injective. 
\end{proof}
\end{lem}
In fact, putting the~$\beta_i$ maps together gives an embedding \[\beta\colon\prod_{i=1}^n\uemb(L_i,B_i)\hookrightarrow\plink(L).\]
The image of $\beta$ is the subspace of $\plink(L)$ where $L_i$ lies in $B_i$. Let $G_L=\prod_{i=1}^n G_i$ be the product of the $G_i$. By Lemma \ref{lem:Snow Globe}, after identifying $\pi_1(\uemb(L_i,B_i))$ with $G_i$, $\beta$ induces an injection \[\beta_*\colon G_L\hookrightarrow \pi_1(\plink(L)).\] 
Taking the product of the forgetful maps $\plink(L)\rightarrow \link(L_i)$, we obtain a homomorphism 
\begin{equation}\label{eq map r}
r\colon \pi_1(\plink(L))\rightarrow \prod_{i=1}^n\pi_1(\link(L_i))=G_L,
\end{equation}
which exhibits $G_L$ as a retract.

\subsection{The Dahm homomorphism and the Fouxe-Rabinovitch subgroup} \label{sec:Dahm-Fouxe-Rabinovitch}
To describe the last subgroup of motions, we recall the homomorphism \[D\colon \pi_1(\plink(L))\rightarrow\Aut(\pi_1(\R^3\setminus L))\cong \Aut(H_1*\cdots*H_k)=\Aut(H_L)\] defined as follows. Consider a family of embeddings $\gamma=\gamma_t$, $t\in[0,1]$, such that $\gamma_0=\gamma_1=L$. Extend this to $\R^3$ via an ambient isotopy $f_t$ with compact support.  Choose now a basepoint $p\in \R^3$ that is fixed by $f_t$ for all $t$.  Thus $(f_1)_*$ defines an automorphism of $\pi_1(\R^3\setminus L,p)$ and we set $D(\gamma)=(f_1)_*$. Although this definition depends on the choice of $p$, one can make it well-defined by choosing a sequence of basepoints along a ray out to infinity.  Since the isotopies are compactly supported, points sufficiently far along the ray will be fixed. There is a canonical way to identify automorphisms based at points along this ray, and in fact along any ray out to infinity.  See Section 4 of \cite{Goldsmith1} for details.

This homomorphism was first defined in the topological setting by Dahm \cite{Dahm}, and then later explored by Goldsmith \cite{Goldsmith1,Goldsmith}, who computed its image when $L$ is an $(np,nq)$ - torus link. The Dahm homomorphism was introduced in the smooth setting by Wattenberg \cite[3.1-2]{Wattenberg1972}.

\begin{lem}\label{lem:Conjugation}
Any automorphism in the image of $D$ sends $H_i$ to a conjugate of itself.  
\end{lem}
\begin{proof} Recall~$H_i=\pi_1(\R^3\setminus L_i)$. Write \[H_L=\pi_1(\R^3\setminus L)=H_1*\cdots*H_r*\overbrace{\Z*\cdots*\Z}^{n-r}\]
where $H_1,\ldots, H_r$ correspond to pieces that are not unknots, and the ($n-r$) $\Z$-factors correspond to unknot pieces. The uniqueness statement of the Grushko decomposition theorem (see \emph{e.g.}~\cite{Stallings}) states that in any maximal free product decomposition of $H_L$, the factors $\{H_1,\ldots, H_r\}$ are determined up to permutation and replacement by conjugates. Hence the fact that $\pi_1(\plink(L))$ preserves labels implies that  any automorphism in the image of $D$ preserves the conjugacy class of each $H_i$, $1\leq i\leq r$.  On the other hand, each of the $\Z$ factors in~$H_L$ can be represented by a loop that goes through the corresponding unknot piece.  The fact that $\pi_1(\plink(L))$ preserves labels means that the free homotopy class of this loop is preserved.  Hence each $\Z$-factor is also taken to a conjugate of itself.
\end{proof}

Certain automorphisms, which we now describe, lie naturally in the image of~$D$.

\begin{defn}Choose an element $g\in H_i$ and a factor $H_j$, where possibly $j=i$.  Define the \emph{partial conjugation of $H_j$ by $g$} to be the automorphism $X(g,H_j)\in \Aut(H_L)$ which sends $h\mapsto ghg^{-1}$ for every $h\in H_j$, and acts as the identity on all other factors $H_k$, $k\neq j$. We refer to $g$ as the \emph{acting element} of $X(g,H_j)$ and to $H_j$ as the \emph{support}. The \emph{Fouxe-Rabinovitch group} $\FR(H_L)$ is the subgroup of $\Aut(H_L)$ generated by the partial conjugations $X(g,H_j)$ where $g\in H_i$ and $i\neq j$.
\end{defn}

We now describe loops in $\plink(L)$ that represent each $X(g,H_j)$, via the Dahm homomorphism~$D$.
\begin{enumerate}
\item When $i\neq j$, $X(g, H_j)$ can be represented geometrically by a family of embeddings $\gamma_t$ that shrink the piece $L_j$ and move it along a path  in the complement of $L_i$ which represents $g\in H_i=\pi_1(\R^3\setminus L_i)$. It is clear that we may choose $\gamma_t$ in such a way that:
\begin{itemize}
    \item If $k\neq j$,   $\gamma_t$ fixes $L_k$ for all $t$.
    \item For every~$t$, $\gamma_t(L_j)$ is a rescaling of $L_j$ followed by a sequence of translations.
\end{itemize}

\item When $i=j$, we can represent $X(g,H_i)$ as follows.  Choose a basepoint $p\in \R^3$ and represent $g\in H_i=\pi_1(\R^3\setminus L_i)$ as a knot $\ell$ in $\R^3$. We can thicken $\ell$ to a solid torus~$\nu(\ell)$ that links $L_i$.  Now consider a family of embeddings $\gamma_t$ which send each point of $L_i$ around the solid torus and back to itself, \emph{i.e.,} such that every point of $L_i$ moves along a loop in $\R^3\setminus \nu(\ell)$ homotopic to the longitude $\ell\subset \nu(\ell)$. Since the basepoint $p$ is fixed in this isotopy, this has the effect of conjugating each element of $H_i=\pi_1(\R^3\setminus L_i)$ by $g=[\ell]$. We can choose~$\nu(\ell)$ such that the whole motion takes place inside~$B_i$, and thus lies in~$G_i$.
\end{enumerate}
Let $\chi(g, L_j)$ in $\pi_1(\plink(L))$ denote the homotopy class of $\gamma_t$ as described in either $(1)$ or $(2)$.  More generally, given $g\in H_i$ and  any subset $A\subseteq \{L_1,\ldots, L_n\}\setminus \{L_i\}$, let $\chi(g,A)$ denote the product $\prod_{L_j\in A} \chi(g,L_j)$. For $j_1\neq j_2$ and not equal to $i$, $\chi(g,L_{j_1})$ and $\chi(g,L_{j_2})$ can be chosen to have disjoint support and thus commute. Thus $\chi(g,A)$ depends only on $A$.

\begin{defn} Define the \emph{Fouxe-Rabinovitch group of $L$} to be the subgroup\[\FRL=\langle \chi(g,L_j)\mid g\in H_i,~ i\neq j\rangle.\]
\end{defn}

In other words, $\FRL$ is the subgroup generated by partial conjugations whose acting elements are disjoint from their support. By definition, the image of $\FRL$ under $D$ is $\FR(H_L)$. Moreover, we have:

\begin{lem}\label{lem - FR injects} The restriction of  $D$ to $\FRL$ is an isomorphism. 
\end{lem}
\begin{proof}
By a result of \cite{FouxeRabinovitch2} (see also \cite{CollinsGilbert}, Proposition 3.1 and Remark (ii) on page 164) the  relations in $\FR(H_L)$ have the following three forms:
\begin{enumerate}
    \item For all $j$ and for $g,g'\in H_i$, $X(g,H_j)X(g',H_j)=X(gg',H_j)$. Here $i=j$ is allowed.
    \item Suppose $g\in H_i$ and $g'\in H_{i'}$. Then $[X(g,H_j),X(g',H_{j'})]=1$ if $j\neq j'$ and $\{i,i'\}\cap \{j,j'\}=\emptyset$. Here $i=i'$ is allowed.

    \item $[X(g,H_j)X(g,H_k),X(g',H_j)]=1$ where $g\in H_i$, $g'\in H_k$ and $i,~j,~k$ are all distinct.
\end{enumerate}
 To prove the lemma, we verify that the chosen generators of $\FRL$ also satisfy these relations. For (1), realise $g,g'$ as loops $\gamma$, $\gamma'$ based at the center of $B_j$. Then the relation in (1) is simply the statement that $\gamma*\gamma'$ is homotopic to $\gamma$ followed by $\gamma'$. For (2), since $j\neq j'$ and neither is equal to $i$ nor $i'$, as we noted above the geometric representatives $\chi(g,L_j)$ and $\chi(g',L_{j'})$ can be chosen to have disjoint support, and thus can be realised simultaneously. Finally, for  (3), since $i$, $j$, and $k$ are distinct we shrink $B_j$ and move it via a sequence of translations into $B_k$.  Then moving~$B_k$ along the loop $\gamma$ for $g\in H_i$ as in the definition of $\chi(g,L_k)$ will also move~$B_j$. Now identify $H_k$ with $\pi_1(B_k\setminus L_k)$. Since we have shrunk down $B_j$ into $B_k$ we are free to move~$B_j$ along the loop $\gamma'\in B_k\setminus L_k$ corresponding to $g'\in H_k$ at any point whilst completing the $\gamma$ path with~$B_k$.  This means that $\chi(g,L_j)\chi(g,L_k)$ commutes with $\chi(g',L_j)$, so (3) is satisfied. Hence $\FRL$ satisfies all relations of $\FR(L)$, which proves that $D$ restricted to $\FRL$ is injective. Since it is surjective on generators, this proves the lemma.
\end{proof}

Recall the map~$r$ from Equation (\ref{eq map r}), from which we obtain a split short exact sequence
\begin{equation}\label{eqn:GL semidirect FR} 1\rightarrow \ker(r)\rightarrow \pi_1(\plink(L))\xrightarrow{r}G_L\rightarrow 1
\end{equation}
From the properties of our chosen representative for $\chi(g,L_j)$ with $g\in H_i$ and $j\neq i$, we deduce that $r(\chi(g,L_j))=1$. In particular, $\FRL\leq \ker(r)$ and $\FRL\cap G_L=\ker(r)\cap G_L=\{1\}$. 

In what follows, we will show that $\ker(r)=\FRL$, and therefore obtain a complete description of $\pi_1(\plink(L))$ as a semidirect product. To accomplish this, in the next section we will analyse which motions of $L$ can be achieved in the complement of a separating system, and show that each of these is generated by a $G_L$ together with a subgroup of $\FRL$.

\subsection{The motion group of $L$ with a separating system.}Consider  an arbitrary embedding $\rho\in \plink(L)$ with separating system $\Sigma$. Let $\pSep_0^{\rho\sqcup \Sigma}$ be the component of $\pSep_0$ containing $\rho\sqcup \Sigma$.   Our goal will be to compute $\pi_1(\pSep_0^{\rho\sqcup \Sigma})$. Recall from Definition~\ref{def-combinatorial type} that the separating system $\Sigma$ defines a rooted labelled tree $T=T_\Sigma$ where each vertex is a component of $\R^3\setminus \Sigma$, labelled by~$i\in \{1,\ldots, n\}$ if it contains the piece $\rho_i$, and each edge is a sphere of  $\Sigma$.  The root is the unbounded component of $\R^3\setminus \Sigma$.

The labels on the rooted tree $T$ define a partial ordering on the set $\{1,\ldots, n\}$ based on proximity to the root. We think of the root as being at the top, with all edges oriented away and downward. For a vertex $v\in T$, denote the set of descendants (which does not include $v$) by $\Desc(v)$. If $v$ is labelled by $i$ then we also write $\Desc(i)=\Desc(v)$.   We single out the direct descendants -- or children -- of vertex $v$ as $\Desc_1(v)=\Desc_1(i)$. In terms of spheres, children of an $i$-vertex are separated by a single sphere from the component of $\R^3\setminus \Sigma$ containing $\rho_i$, and by at least one sphere from the unbounded component.  Each vertex $v$ defines a descending branch of $T$ below $v$. Let $L(v)$ be the subset of basepoint pieces~$L_j$ such that~$j\in\{1,\ldots,n\}$ occurs as a label of this branch. Since $\Sigma$ is an essential separating system, $L(v)$ is non-empty.  

\begin{defn}\label{def:Motion group of T}Let $\Sigma$ be an essential separating system for $\rho$ and $T_\Sigma$ its dual tree. We define the \emph{group of partial conjugations supported on $T_\Sigma$} to be 
$$\chi(T_\Sigma)=\langle\chi(g, L(v))\mid v\in \Desc_1(L_i) \text{ and }g\in H_i\text{ for }1\leq i\leq n\rangle$$ 
Define the \emph{motion group of $T_\Sigma$} to be the group $G(T_\Sigma)$ generated by $G_L$ and $\chi(T_\Sigma)$.
\end{defn}

\begin{prop}\label{prop:Bubble Boyz} For any separating system $\Sigma$,
$\pi_1(\pSep_0^{\rho\cup \Sigma})\cong G(T_\Sigma)$. Moreover,   $\chi(T_\Sigma)\leq G(T_\Sigma)$ is normalised by $G_L$.
\end{prop}

\begin{proof}
We prove this by induction on the number of pieces. For the base case $L=L_1$, $\plink(L_1)=\link(L_1)$, and by definition $\pi_1(\link(L_1))=G_1$. 

Now suppose $L=L_1\sqcup\cdots\sqcup L_n$ with $n>1$. Let $r$ be the root and $\Desc_1(r)=\{v_1,\ldots, v_k\}$. Each $v_j$ together with its descendants induces a rooted subtree $T_j$ of $T$, where $v_j$ is the root. Since $\Sigma$ is a separating system, each $T_j$ has label set containing at least one $i$-labelled vertex for $i\in \{1,\ldots , n\}$.

Furthermore, each $v_j\in \Desc_1(r)$ defines an edge~$\{r,v_j\}$ corresponding to a sphere $S_j$. Let $B_j$ be the ball bounded by $S_j$.  Let $\rho(j)$ be the sublink of $\rho$ 
which is the union of all pieces in the interior of~$B_j$. Since the interior of $B_j$ is homeomorphic to $\R^3$, we can regard the union of spheres in the interior of $B_j$ as a separating system $\Sigma_j$ for $\rho(j)$, with dual tree $T_j$.  Denote by $\pSep_0^{\rho(j)\sqcup \Sigma_j}$ the connected component of $\pSep_0$ for $\rho(j)$ containing $\rho(j)\sqcup \Sigma_j$. Since the root $r$ of~$T$ cannot be univalent, the number of pieces in $\rho(j)$ is less than~$n$. Therefore, by induction, $\pi_1(\pSep_0^{\rho(j)\sqcup \Sigma_j})\cong G(T_{\Sigma_j})$.

Now consider the map which forgets the link pieces and spheres in the interior of each $B_j$. The fiber is homeomorphic to the product $\prod_{j=1}^k\pSep_0^{\rho(j)\sqcup \Sigma_j}$.  There are two cases, depending on whether the root $r$ is labelled by some $i\in \{1,\ldots,n\}$ or not.  First consider the case in which $r$ is labelled by $\emptyset$.  In this case, we obtain a fibration
\[\prod_{j=1}^k\pSep_0^{\rho(j)\sqcup \Sigma_j}\hookrightarrow \pSep_0^{\rho\sqcup\Sigma}\rightarrow \uemb(B_1\sqcup\cdots\sqcup B_k,\R^3). \]
By Lemma \ref{lem:equiv between balls and config}, the base space is homotopy equivalent to $\pconf_k(\R^3)$, which is simply connected.   Hence $\pi_1(\pSep_0^{\rho\sqcup\Sigma})\cong \prod_{j=1}^kG(T_{\Sigma_j})$. By considering generators, we see the latter is equal to $G(T_\Sigma)$, proving the proposition in this case. 

Now suppose $r$ is labelled by some $i\in \{1,\ldots , n\}$. Relabelling if necessary, we may assume that $r$ is labelled by $n$.  In this case we have the following fibration
\begin{equation}\label{eqn - root labelled fiber sequence}
\prod_{j=1}^k\pSep_0^{\rho(j)\sqcup \Sigma_j}\hookrightarrow \pSep_0^{\rho\sqcup\Sigma}\rightarrow \uemb(L_n\sqcup B_1\sqcup\cdots\sqcup B_k,\R^3).
\end{equation}

We claim that the fundamental group of the base space is isomorphic to the group $(\prod_{j=1}^k H_n)\rtimes G_n$.  By Lemma \ref{lem:Balls Complement}, $\uemb(L_n\sqcup B_1\sqcup\cdots\sqcup B_k,\R^3)$ is homotopy equivalent to $\uemb(L_n\sqcup \{p_1,\ldots,p_k\},\R^3)$, where $\{p_1,\ldots, p_k\}$ are $k$ distinct labelled points. The latter fits into a fiber bundle 
\[\pconf_k(\R^3\setminus L_n)\hookrightarrow \uemb(L_n\sqcup \{p_1,\ldots,p_k\},\R^3)\rightarrow \link(L_n).\]

The fundamental group of $\link(L_n)$ is $G_{n}$ by definition.  Since a point has codimension 3 in $\R^3\setminus L_n$, it follows that $\pi_1(\pconf_k(\R^3\setminus L_n))\cong \prod_{j=1}^kH_n$. Hence on the level of $\pi_1$
we obtain a short exact sequence
\begin{equation*}1\rightarrow \prod_{j=1}^kH_n\rightarrow \pi_1(\uemb(L_n\sqcup \{p_1,\ldots,p_k\},\R^3))\rightarrow G_{n}\rightarrow 1.
\end{equation*} This sequence is split, since the corresponding map of spaces has a section up to homotopy, defined by putting $L_n$ in one half space and $\{p_1,\ldots,p_k\}$ in the other. This proves the claim.

We return to Equation (\ref{eqn - root labelled fiber sequence}). By induction, the fundamental group of each $\pSep_0^{\rho(j)\sqcup \Sigma_j}$ is generated by $G(T_{\Sigma_j})$, while $\pi_1(\emb(L_n\sqcup B_1\sqcup\cdots\sqcup B_k,\R^3))\cong (\prod_{j=1}^k H_n)\rtimes G_n$. Moreover, the product $\prod_{j=1}^k H_n$ is generated by a family of embeddings that sends $B_j$ along some loop in the complement of $L_n$ representing an element $g_n\in H_n= \pi_1(\R^3\setminus L_n)$.  This induces the element $\chi(g_n,L(v_j))$ on the level of fundamental groups. (Recall $L(v)$ is the set of pieces~$L_j$ such that~$j$ is a label in the descending branch of~$T$ defined by $v$.) Hence $G(T_\Sigma)$ is generated as claimed.

Finally, we prove that $\chi(T_\Sigma)$ is normalised by $G_L$. By induction, the partial conjugations $\chi(T_{\Sigma_j})$  are normalised by each $G_{p}$ if $L_p$ is a piece in $\rho(j)$.  Moreover, for each $j'\neq j$, elements of $G(T_{\Sigma_j})$ commute with $G(T_{\Sigma_j'})$ because they have disjoint support as automorphisms. In particular, if $L_q$ is a piece in $\rho(j')$, then $G_q$ commutes with $\chi(T_{\Sigma_j})$. Lastly, the existence of the semidirect product decomposition $(\prod_{j=1}^k H_n)\rtimes G_n$ now implies that the added partial conjugations are also normalised by $G_L$, since $G_{i}$ commutes with $\prod_{j=1}^k H_n$ for each $i\neq n$.
\end{proof}
\begin{rem}\label{rem:Normalise}
By Lemma \ref{lem:ComType}, all combinatorial types of separating system are realised for a given $\rho$. Since every element of $\FRL$ arises in some separating system,  it follows from Proposition~\ref{prop:Bubble Boyz} that $G_L$ normalises $\FRL$. It follows from the exact sequence in (\ref{eqn:GL semidirect FR}) that $\pi_1(\plink(L))$ contains $\FRL\rtimes G_L$ as a subgroup. 
\end{rem}

\subsection{Finishing the proof of Theorem \ref{abcthm - motion group}} Recall that our goal is to show that $\ker(r)=\FRL.$ We say that an automorphism $\varphi$ of $H_L$ \emph{preserves the free splitting} $H_1*\cdots*H_n$ if $\varphi(H_i)=H_i$ for each $i$.  We claim that to prove $\ker(r)=\FRL$ it suffices to show:

\begin{lem}\label{prop: Fix Splitting in GL}
If $g\in \pi_1(\plink(L))$ and $D(g)$ preserves the free splitting $H_1*\cdots* H_n$, then $g\in G_L$.
\end{lem}
 
Assuming the lemma, let us prove:

\begin{prop}\label{prop - ker(r)}Let $r\colon \pi_1(\plink(L))\rightarrow G_L$ be as in Equation (\ref{eq map r}).
    Then $\ker(r)=\FRL$ and in particular, $\pi_1(\plink(L))\cong \FRL\rtimes G_L$.
\end{prop}
\begin{proof}
Suppose $g\in \ker(r)$.  By Lemma \ref{lem:Conjugation}, $D(g)$ sends each factor $H_i$ of $H_L$ to a conjugate of itself. Since $\FRL\leq \ker(r)$, there exists an element $h\in \FRL $ such that $D(gh^{-1})$ preserves the free splitting $H_1*\cdots*H_n$.  But then by Lemma \ref{prop: Fix Splitting in GL}, we see that $gh^{-1}\in \ker(r)\cap G_L=\{1\}$. This means that $g=h\in \FRL$, proving that  $\FRL=\ker(r)$. The semidirect product follows since $r$ is split.
\end{proof}

It remains to prove the lemma. For this, we will use the homotopy equivalence from Theorem \ref{abcthm - sep equiv} together with the description of the motion group of $L$ with a separating system from Proposition \ref{prop:Bubble Boyz}.

\begin{proof}[Proof of Lemma \ref{prop: Fix Splitting in GL}]

By Corollary~\ref{cor - PSep Homotopy equivalence} to Theorem \ref{abcthm - sep equiv},~$\plink(L)\simeq |\pSep_\bt|$. We can therefore lift any class in~$\pi_1(\plink(L))$ to~$|\pSep_\bt|$, and moreover (by \cite[Lemma 2.1]{EbertRandalWilliams}) we can find a representative $\gamma\colon S^1\rightarrow \plink(L)$ such that the image of $\gamma$ lifts to the 1-skeleton $|\pSep_\bt|^{(1)}$ (as defined in \cite[Equation (1.16)]{EbertRandalWilliams}). In particular we obtain a subdivision of $[0,1]$ into intervals $I_0J_1\cdots I_{p-1}J_{p}I_{p}$ such that $\gamma(I_k)$ lifts to $\pSep_0 $ and $\gamma(J_l)$ lifts to $\pSep_1 $ for all $k,l$ (by an abuse of notation, we will omit $\gamma$ and think of the $I_k,~J_l$ as paths themselves). Since our basepoint (from Example~\ref{example - basepoint SS}) lies in $\pSep_0^{L\sqcup \Sigma_L}$, both $I_0$ and $I_p$ lift to $\pSep_0^{L\sqcup \Sigma_L}$. We will show that $g$ can be written as a product of partial conjugations and elements of $G_L$.  Since $G_L$ normalises the partial conjugations by Remark \ref{rem:Normalise}, our assumption that $D(g)$ preserves the free splitting $H_1*\cdots* H_n$ will imply the lemma.

For each $1\leq l\leq p$, $J_l$  connects two components of $\pSep_0$. Let $\rho_{l}\in \plink(L)$ be the embedding determined by $J_l(0)$.  At $J_l(0)$, there are two separating systems $\Sigma_{l-1}$ and $\Sigma_l$ for $\rho_l$.   On the interior of $J_l$, both separating systems are disjointly embedded with the link and vary only by isotopy. Denote by $\pSep_1^{\rho_l\sqcup\Sigma_{l-1}\sqcup\Sigma_l}$ the connected component of $\pSep_1$ containing $\rho_l\sqcup\Sigma_{l-1}\sqcup\Sigma_l$. 

{\bf Claim:}
On each $J_l$ we may assume that the embeddings of $L$, and the associated separating systems~$\S_{l-1}$ and $\S_l$ are fixed.

\noindent\emph{Proof of claim.} There's a forgetful map $f:\pSep_1^{\rho_l\sqcup\Sigma_{l-1}\sqcup\Sigma_l} \to \pSep_0^{\rho_l\sqcup\Sigma_{l}}$, given by forgetting the embedding of~$\S_{l-1}$. We replace the segments~$J_l$ and~$I_{l+1}$ with segments~$J'_l$ and~$I'_{l+1}$ as follows. Let~$J'_l$ be the path in ~$\pSep_1^{\rho_l\sqcup\Sigma_{l-1}\sqcup\Sigma_l}$ where~$L$, remains fixed at~$J_l(0)$ and~$\Sigma_{l-1}$ and $\Sigma_l$ are fixed on the interior of $J_l$, with the weight shifting linearly from $\Sigma_{l-1}$ to $\Sigma_l$. Let~$J_l|_{L\cup \S_l}$ be the isotopy of $L\cup{\S_l}$ which occurs in the interior of $J_l$. We concatenate~$f(J_L)=J_l|_{L\cup \S_l}$ with the path~$I_{l+1}$ to get~$I'_{l+1}$. 
By homotoping $\gamma$, we may replace the segment~$J_lI_{l+1}$ with~$J'_lI'_{l+1}$. Relabelling $J'$ to $J_l$ and $I'_{l+1}$ to~$I_{l+1}$ proves the claim.\qed

We next consider the contribution of each interval $I_k$. Denote the link embedding at $I_k(0)$ (resp. $I_k(1)$) by $\rho_k$ (resp. $\rho_k'$), with separating system $\Sigma_k$ (resp. $\Sigma_k'$).  Choose a path $\alpha_k$ from $\rho_k$ to $L$.  For $I_0(0)$ and $I_p(1)$ we choose these paths to be constant at $L$. By the claim, $\rho_{k-1}'=\rho_k$ for $1\leq k\leq p$, and the separating systems $\Sigma_{k-1}'$ and $\Sigma_k$ are disjointly embedded. Therefore we extend $\alpha_k$ to both separating systems by isotopy extension. The path $\overline{\alpha_k}*I_k*\alpha_{k+1}$ for $0\leq k\leq p-1$ lifts to an element of $\pi_1(|\pSep_\bt|)$, and the concatenation \[(I_0\alpha_1)(\overline{\alpha_1}J_1\alpha_1)(\overline{\alpha_1}I_1\alpha_2)\overline{\alpha_2}\cdots\alpha_p(\overline{\alpha_p}J_{p}\alpha_p)(\overline{\alpha_p}I_p)\] is homotopic to $\gamma$.  By the claim, each term $\overline{\alpha_l}J_l\alpha_l$ is trivial in $\pi_1(\link(L))$, since $J_l$ is stationary on the link.  

On the other hand, each term $\overline{\alpha_k}I_k\alpha_{k+1}$ can be regarded as an element of the motion group $\pi_1(\pSep_0^{\rho_k\sqcup \Sigma_k})$. By Proposition \ref{prop:Bubble Boyz}, this is an element $g_p$ of $G(T_{\Sigma_k})$ which is a product of elements of $\FRL$ and elements of $G_L$. Since $G_L$ normalises $\FRL$ (Remark \ref{rem:Normalise}), we can write this as a product $g_k=h_kt_k$, where $h_k\in \FRL$ and $t_k\in G_L$.  Doing this for all~$ 0\leq k \leq p$ gives \begin{align*}g=g_0\cdots g_p&=(h_0t_0)(h_1t_1)\cdots(h_pt_p)\\
&=(h_0h_1'\cdots h_p')(t_0\cdots t_p)
\end{align*}
where on the last line we used the fact that $G_L$ normalises $\FRL$ again. Since $D(G_L)$ preserves the free splitting $H_1*\cdots* H_n$, the product $D(t_0\cdots t_p)$ also does.  Now since $D(g)$ preserves the free splitting by assumption we conclude that $D(h_0h_1'\cdots h_p')$ preserves the free splitting as well. Under $D$, each element in $\FRL$ conjugates some $H_j$ by $g\in H_i$ where $i\neq j$.  Hence the only way that $D(h_0h_1'\cdots h_p')$ preserves the free splitting is if $D(h_0h_1'\cdots h_p')=1$. By Lemma \ref{lem - FR injects}, this implies $h_0h_1'\cdots h_p'=1\in \FRL$, hence $g=t_0\cdots t_p\in G_L$, as desired.
\end{proof}

We are now ready to revisit the unlabelled case and prove Theorem \ref{abcthm - motion group}:

\begin{restate}{Theorem}{abcthm - motion group}
$\pi_1(\link(L))$ is isomorphic to $(\FRL\rtimes G_L)\rtimes P_L$.
\end{restate}
 \begin{proof}

By Lemma \ref{lem:Permute Links}, the homomorphism $\eta\colon\pi_1(\link(L))\rightarrow P_L$ induces a short exact sequence \[1\rightarrow \pi_1(\plink(L))\rightarrow \pi_1(\link(L))\rightarrow P_L\rightarrow 1.\]
By Proposition \ref{prop - ker(r)}, $\pi_1(\plink(L))\cong \FRL\rtimes G_L$.
We now check that the lifts of elements of $P_L$ described in the proof of Lemma~\ref{lem:Permute Links} split the surjection $\eta$. To show this is a homomorphism we check it respects the relations. Let $\gamma=\sigma_1\cdots \sigma_q$ be a product of these lifts which induces the trivial permutation.  Since each $\sigma_i$ acts as a sequence of translations, we have that $r(\gamma)=1$.  On the other hand, $\gamma$ also preserves the free splitting $H_1*\ldots*H_n$.  Hence $\gamma=1\in \pi_1(\plink(L))$ by Lemma \ref{prop: Fix Splitting in GL}.
\end{proof}
\subsection{Example: $H$-trivial links}Let $H$ be the Hopf link.  By work of Goldsmith, the topological motion group is isomorphic to the quaternion group $Q_8$, since $H$ is isomorphic to the torus link T(2,2) \cite{Goldsmith}. In \cite{DamianiKamada}, Damiani and Kamada showed that the same holds for the \emph{ring motion group} of $H$, which is the fundamental group of the round unparametrised embedding space.
In upcoming work, we show that the smooth and round embedding spaces are homotopy equivalent and, independently to the Daminai-Kamada result, obtain that $\pi_1(\link(H))\cong Q_8$. Using this and Theorem \ref{abcthm - motion group}, we can compute the fundamental group of $\link(L)$ when $L$ is a disjoint union of an $n$-component unlink and $m$ Hopf links. Following \cite{DamianiKamada}, we call such a link \emph{$H$-trivial}.

\begin{exam}[H-trivial links]\label{example - Htrivial} Let $L=H_{n,m}$ be a disjoint union of an $n$-component unlink and $m$ Hopf links. Then we have \[\pi_1(\R^3\setminus L)=\overbrace{\Z*\cdots*\Z}^{n}*\overbrace{\Z^2*\cdots*\Z^2}^{m}. \] We can permute each of the unlink components and each of the Hopf links, so $P_L\cong S_n\times S_m$.  Each unknot contributes a $\Z/2$ factor to $G_L$, while each Hopf link contributes a $Q_8$ factor.  Thus $G_L\cong (\Z/2)^n\times ( Q_8)^m$. Therefore we obtain
\[\pi_1(\link(H_{n,m}))\cong \left(\FR(\overbrace{\Z*\cdots*\Z}^{n}*\overbrace{\Z^2*\cdots*\Z^2}^{m})\rtimes ((\Z/2)^n\times ( Q_8)^m)\right)\rtimes (S_n\times S_m).\]
For $m=0$ this recovers the result for unlinks. Since $\Z$ and $\Z^2$ are finitely presented, so is $\mathcal{FR}(H_{n,m})$. An explicit presentation can be deduced from the relations in the proof of Lemma \ref{lem - FR injects}. Hence, $\pi_1(\link(H_{n,m}))$ is finitely presented as well. 
\end{exam}

From personal correspondence, we believe that Damiani, Kamada and Piergallini have independently computed a presentation for this motion group.

\subsection{Passing to~\texorpdfstring{$S^3$}{the 3-sphere} }
In this section we discuss the effect on homotopy groups, and specifically $\pi_1$, when passing to embeddings in $S^3$ instead of $\R^3$.  Consider the space of embeddings $\uemb(L\sqcup \{*\},S^3)$, where $\{*\}$ is a disjoint point.  Then, regarding  $\R^3$ as $S^3\setminus \{*\}$, we identify  $\uemb(L,S^3\setminus \{*\})\cong \uemb(L,\R^3)=\link(L)$.  We then obtain the following two fibrations.\[
\xymatrix@C=5mm{&\uemb(\{*\},S^3\setminus L)=S^3\setminus L\ar[d]&\\
\uemb(L,S^3\setminus\{*\})\cong\link(L)\ar[r]&\uemb(L\sqcup \{*\},S^3)\ar[r]\ar[d]&\uemb(\{*\},S^3)=S^3\\
&\uemb(L,S^3)&
}
\]
Since $S^3$ is 2-connected, the long exact sequence for the horizontal fibration yields an isomorphism $\pi_1(\link(L))\cong \pi_1(\uemb(L\sqcup\{*\},S^3))$. Combined with the vertical fibration we then have a short exact sequence
\begin{equation}\label{eqn - conjugation les}
    \pi_2(\uemb(L,S^3))\rightarrow \pi_1(S^3\setminus L)\rightarrow \pi_1(\link(L))\rightarrow \pi_1(\uemb(L,S^3))\rightarrow 1.
\end{equation}

Recall that for any group $G$ the conjugation action of $G$ on itself induces a homomorphism $G\rightarrow \Aut(G)$ whose image is the the subgroup of \emph{inner automorphisms} $\Inn(G)$. The kernel of this homomorphism is the center of $G$.  
\begin{prop}\label{prop - inner auts s3}
The image of $\pi_1(\uemb(\{*\},S^3\setminus L))$ in $\pi_1(\link(L))$ is isomorphic to the group of inner automorphisms of $\pi_1(S^3\setminus L)$.  In particular, the image of $\pi_2(\uemb(L,S^3))$ is the center of $\pi_1(S^3\setminus L)$.
\begin{proof} Regard $\{*\}\in S^3\setminus L$ as a basepoint for the fundamental group. An element of $\pi_1(\uemb(\{*\},S^3\setminus L))$  drags the basepoint around a loop $\gamma$ representing a class in $\pi_1(S^3\setminus L)$, which changes the identification of the fundamental group by conjugation by $[\gamma]$. The image of $\pi_1(\uemb(\{*\},S^3\setminus L)$ is therefore the subgroup of inner automorphisms which lies in $\FRL\leq \pi_1(\link(L))$. The statement about the image of $\pi_2(\uemb(L,S^3))$ follows from exactness of Equation (\ref{eqn - conjugation les}).
\end{proof}
\end{prop}

Inputting Proposition~\ref{prop - inner auts s3} with Equation~(\ref{eqn - conjugation les}), Theorem~\ref{abcthm - motion group}, and the fact that the image of $\pi_1(\uemb(\{*\},S^3\setminus L)$ is normal now implies the following corollary.

\begin{restate}{Corollary}{abccor - motion group S3} Let $H_L=\pi_1(S^3\setminus L)$. Then 
\[\pi_1(\uemb(L,S^3))\cong ((\FRL\rtimes G_L)/\Inn(H_L))\rtimes P_L.\]
where $\Inn(H_L)$ is the group of inner automorphisms of $H_L$.
\end{restate}

We end this section with some remarks on Theorem \ref{abcthm - motion group} and Corollary \ref{abccor - motion group S3}.   Recall that if $L$ has a single piece, then $S^3\setminus L$ is aspherical.  In particular, $H_L=\pi_1(S^3\setminus L)$ is finitely presented, torsion-free and has cohomological dimension at most 2.  If $H_L$ is abelian, then it is isomorphic to either $\Z$ or $\Z^2$.  The former occurs exactly when $L$ is the unknot $U$, and the latter exactly when $L$ is the Hopf link $H$.

If $L$ has at least two pieces, $H_L$ is a nontrivial free product, hence its center is trivial. In this case $\FRL\leq \FRL\rtimes G_L $ is always infinite.  $\FRL/\Inn(H_L)$ is also infinite, unless $L$ is the two component unlink $U\sqcup U$ or a disjoint union of two Hopf links $H\sqcup H$.  If $L$ has a single piece, then $P_L$ is trivial by definition and $\FRL=\Inn(L)$, so $\pi_1(\uemb(L,S^3))$ is equal to $G_L/\FRL$.  On the other hand, in this case $\FRL$ is infinite if and only if $H_L$ is not abelian. By the preceding discussion, $\pi_1(\link(L))$ is thus finite if and only if $L=U$ or $L=H$, and in this case $\pi_1(\link(U))$ is $\Z/2$ while $\pi_1(\link(H))$ is $Q_8$.

\section{Contractibility of~\texorpdfstring{$|\Sep(\rho)^\delta_\bt|$}{the fiber}}\label{section - contractible}

The aim of this section is to prove Theorem~\ref{thm-discrete fiber is contractible}. To do this, we introduce a larger contractible space, and show that the inclusion of~$|\Sep(\rho)^\delta_\bt|$ into this space is a weak equivalence. We also prove Lemma~\ref{lem - discrete missing sep is contractible}. After writing the argument in this section we discovered that work of Mann and Nariman \cite[Section 3]{MannNariman}, and an unpublished manuscript of Hatcher \cite{Hatcherwrongway}, address similar complexes, by similar means. However, as our argument is considerably more detailed we felt it worth including in full.

\begin{defn}
Fix~$\rho\in \link(L)$. The semi-simplicial space $\TSeprho_\bt$ is defined as follows:
\begin{itemize}
    \item The space $\TSeprho_0$ of 0-simplices is equal to~$\Sep(\rho)^\delta_0$, \emph{i.e.,}~it is given by the subspace
    \[\{(\gamma, \Sigma) \in \ESS(L)\text{ such that }\gamma=\rho\}\] 
    equipped with the discrete topology.
    \item The space $\TSep_p$ of~$p$-simplices is the subspace of~$(\TSeprho_0)^{p+1}$ consisting of ordered~$(p+1)$ tuples~$(\rho,\Sigma_i)$ such that for all~$0\leq i<j\leq p$, either~$\Sigma_i\cap \Sigma_j= \emptyset$ or the spheres of~$\Sigma_i$ and~$\Sigma_j$ intersect transversely. 
    \item Face maps $\p_p^{i}\colon \Sep(\rho)_p\rightarrow \Sep(\rho)_{p-1}$ for $0\leq i\leq p$ are given by forgetting the~$i$th entry in the tuple.
\end{itemize}
\end{defn}

There is an inclusion~$\iota_p:\Sep(\rho)^\delta_p\hookrightarrow \TSeprho_p$, which induces an inclusion on geometric realisations $$|\iota_\bt|:|\Sep(\rho)^\delta_\bt|\hookrightarrow |\TSeprho_\bt|.$$ In this section we will
\begin{enumerate}[(a)]
    \item show $|\TSeprho_\bt|$ is contractible, and
    \item show~$|\iota_\bt|$ is a homotopy equivalence.
\end{enumerate}

\begin{lem}\label{lemma-existence of transverse system}
For a finite set~$J$ let~$\{\eta_j\}_{j\in J} \in~\bigoplus_{j\in J}\uemb(S^2,\R^3\setminus\rho)$ be a collection of embedded spheres in~$\R^3\setminus\rho$, with index set~$J$. Then there exists a separating system $\S \in \TSeprho_0$ with image transverse to every~$\eta_j$.
\begin{proof}
This follows from the residuality of transverse embeddings~\cite[Chapter 3, Theorem 2.1]{Hirsch}, which applied to our scenario shows that the subspace of separating systems transverse to a given embedded sphere~$\eta_j$ is residual in each connected component of $\TSeprho_0$. Since a countable intersection of residual subsets is residual, and~$J$ is finite and therefore countable, it follows that the subset of~${\S} \in \TSeprho_0$ with image pairwise transverse to the collection $\{\eta_j\}_{j\in J}$ is not only non-empty but dense.
\end{proof}
\end{lem}

\begin{rem}
In fact, we can find a separating system for any isotopy class with a prescribed combinatorial type, since these correspond to different connected components of~$\TSeprho_0$ for any fixed $\rho$. This won't be relevant in our proof.
\end{rem}

Recall from~\cite{EbertRandalWilliams} that points in~$|X_\bt|$ have the form~$[x,\ul{t}]$, where for some $p\geq 0$,~$x\in X_p$,~$t\in \Delta^p$, and a quotient is taken which identifies faces in the appropriate way. Recall that
$$\Delta^p=\{\ul{t}=(t_0,\ldots, t_p)\in \R^{p+1} \,|\, \sum_{i=0}^{p}t_i =1\}$$
with the subspace topology.
$|X_\bt|$ is then topologised by taking the product topology on~$\sqcup_{p\geq 0} X_p\times \Delta^p$ for each~$p$, and then the quotient topology upon gluing faces.

Via the above, we can think of a point in~$|\TSeprho_\bt|$ as a tuple~$(\S_0,\ldots, \S_p)\in \TSeprho_p$, where each~$\S_i$ has a \emph{weight}~$t_i$ between 0 and 1 associated to it, and $\sum_{i=0}^{p}t_i =1$. Passing to a face of the simplex corresponds to one of the~$t_i$ becoming zero, at which point the associated~$\Sigma_i$ is deleted from the tuple -- that is, we pass to $\TSeprho_{p-1}$. The weights can be thought of as `opacity' of the separating systems. This discussion also holds for~$|\Sep(\rho)^\delta_\bt|$, replacing $\TSeprho_p$ with $\Sep(\rho)^\delta_p$.

Let~$Y$ be a triangulated space. Then for any simplicial map~$F:Y\to |\TSeprho_\bt|$ we define the set
\[\im(F)_0=\{[\S,1]\in |\TSeprho_\bt| \, s.t. \, \exists \,p\in Y^{(0)}\text{ with } F(p)=[\S,1] \}
\]
\emph{i.e.,}~$\S\in \im(F)_0$ has weight 1, and is the image of a 0-simplex in~$Y$. Since the image of a~$p$-simplex is determined by its vertices,~$\im(F)_0$ is exactly the set of separating systems in the image of the map~$F$. 

\begin{prop}\label{prop-Temb is contractible}
$|\TSeprho_\bt|$ is contractible.
\begin{proof}
Consider a simplicial map~$f\colon S^k \to |\TSeprho_\bt|$ from a triangulation of the sphere~$S^k$, $k\geq 0$.  Since~$S^k$ is finite dimensional and compact, it follows that~$\im(f)_0=\{\Sigma_i\}_{i\in I}$ for~$I$ a finite set indexing~$(S^k)^{(0)}$. This is a point in $\sqcup_{j\in J}\uemb(S^2,\R^3\setminus\rho)$ for some finite set~$J$, \emph{i.e.,}~the image of a finite number of embedded spheres (individually embedded, not as a set) in $\R^3\setminus\rho$. By Lemma~\ref{lemma-existence of transverse system} there exists a separating system ${\S} \in \TSeprho_0$ with image transverse to each of these spheres. Coning off the triangulation of~$S^k$ to a triangulation of~$D^{k+1}$ with cone point~$c$, the map~$f$ extends to a map~$F:D^{k+1}\to |\TSeprho_\bt|$ such that~$F(c)={\Sigma}$ with weight 1. Thus $|\TSeprho_\bt|$ is contractible as required.
\end{proof}
\end{prop}

We now digress slightly and prove Lemma~\ref{lem - discrete missing sep is contractible}. Recall $\Sep(\rho, \hat{\Sigma})^\delta_p$ is the subspace of $\Sep(\rho)^\delta_p$ where the~$p+1$ mutually disjoint separating systems also do not intersect with a given  (not necessarily essential) separating system~$\hat{\S}$.
\begin{proof}[Proof of Lemma~\ref{lem - discrete missing sep is contractible}]
Consider a simplicial map~$f$ from a triangulation of~$S^k$ into $|\Sep(\rho, \hat{\Sigma})^\delta_\bt|$. Let~$(-1,1)\times \hat{\Sigma}\subset\R^3\setminus \rho$ be a tubular neighbourhood of~$\hat{\Sigma}$, such that~$\hat{\S}=(0\times \hat{\Sigma})\subset\R^3$. Then there exists an~$\epsilon>0$ such that~$(\epsilon \times \hat{\Sigma})$ is disjoint from all separating systems in~$\im(f)_0$, since being disjoint from a finite number of separating systems is an open condition. Let~$\S'\subset (\epsilon \times \hat{\Sigma})$ be a sub-collection of the spheres of~$(\epsilon \times \hat{\Sigma})$ such that~$\S'$ is an essential separating system for~$\rho$. Then, as in the proof of Proposition~\ref{prop-Temb is contractible}, $\S'$ acts as a cone point for~$\im(f)$, and thus~$f$ extends to a map~$F: c*S^k\cong D^{k+1}\to |\Sep(\rho, \hat{\Sigma})^\delta_\bt| $ such that~$F(c)={\Sigma'}$ with weight 1. Therefore $|\Sep(\rho, \hat{\Sigma})^\delta_\bt| \simeq *$.
\end{proof}

Recall~$\iota_p:\Sep(\rho)^\delta_p\hookrightarrow \TSeprho_p$ induces $|\iota_\bt|:|\Sep(\rho)^\delta_\bt|\hookrightarrow |\TSeprho_\bt|,$ and our aim was to show
\begin{enumerate}[(a)]
    \item show $|\TSeprho_\bt|$ is contractible (Proposition \ref{prop-Temb is contractible}), and
    \item show~$|\iota_\bt|$ is a homotopy equivalence.
\end{enumerate}
So, to prove Theorem \ref{thm-discrete fiber is contractible}, it remains to show that~$|\iota_\bt|$ is a homotopy equivalence. We show the relative homotopy groups vanish. Consider a representative~$f$ of a class in~$\pi_k(|\Sep(\rho)^\delta_\bt|)$. Then, by Proposition~$\ref{prop-Temb is contractible}$, this extends to a map~$F$ from~$D^{k+1}$ to $|\TSeprho_\bt|$ as shown below.

\begin{equation}\label{eq-pair of maps}
    \xymatrix{f: S^{k}\ar[r]\ar@{^{(}->}[d]_{\partial}& |\Sep(\rho)^\delta_\bt|\ar@{^{(}->}[d]^{|\iota_\bt|}\\
 F: D^{k+1}\ar[r]& |\TSeprho_\bt|
 }
\end{equation}
 
 In the following, let a \emph{pair of maps} be a pair~$(F,f)$ as in the above diagram. 
 Level-wise the right-hand terms are endowed with the discrete topology, so simplicial approximation applies and we can assume~$f$ and~$F$ are both simplicial maps, from triangulations of the sphere and disk respectively.
 Now, the relative homotopy groups vanish if we can homotope a pair of maps~$(F,f)$ to a pair~$(F',f')$ such that~$\im(F')\subset |\Sep(\rho)^{\delta}_\bt|$.  Roughly, our homotopy will occur in steps such that at each step the number of intersections in the image is reduced until we obtain the required pair. A transverse intersection is removed via a surgery move on spheres in a separating system. However it may be the case that we wish to remove a non-transverse intersection, and so the first step is to replace these with transverse intersections.
 
We therefore first introduce a complexity which counts non-transverse intersections. Let~$(F,f)$ be a pair of maps, let $p$ and~$q$ be in~$(D^{k+1})^{(0)}$, and let~$\Sigma_p$ and~$\Sigma_q$ in~$\im(F)_0$ correspond to~$F(p)$ and~$F(q)$, respectively. If~$\S_p\cap \S_q\neq \emptyset$ in~$\R^3$ then the intersection is a union of embedded circles in~$\R^3$. We distinguish two types of intersections.
\begin{enumerate}[(i)]
    \item \emph{Ghost intersections:} Intersections between spheres in~$\Sigma_p$ and~$\Sigma_q$ where $p\notin \lk_{D^{k+1}}(q)$ ($\lk_X(v)$ denotes the link of vertex~$v$ in simplicial complex~$X$).
    \item \emph{Real intersections:} Intersections between spheres in~$\Sigma_p$ and~$\Sigma_q$ where $p\in \lk_{D^{k+1}}(q)$.
\end{enumerate}

By the definition of $|\TSeprho_\bt|$, all real intersections in~$\im(F)_0$ are transverse. It is not the case that the ghost intersections are transverse, but the next lemma asserts we can always replace our pair with one for which this is true. 

\begin{defn}\label{defn-tangential complexity}
Let~$(F,f)$ be a pair of maps as in Equation~(\ref{eq-pair of maps}). For~$\Sigma\in \im(F)_0$, let~$tg_F(\Sigma)$ be 0 if $\Sigma$ intersects all~$\Sigma'$ in $\im(F)_0\setminus \S$ transversely, and 1 otherwise. 
Define the \emph{tangential complexity} of~$F$ to be, 
$$tc^F=\sum_{\S\in \im(F)_0}tg_F(\Sigma).$$
Let the tangential complexity of~$f$ be analogously described, \emph{i.e.,}
$$tc^f=\sum_{\S\in \im(f)_0} tg_f(\Sigma).$$
\end{defn}

In words,~$tc^F$ counts the number of separating systems in the image of~$F$ which have some non-transverse intersections with other separating systems in the image of~$F$. These are necessarily ghost intersections. Note that since $\im(f)_0\subseteq \im(F)_0$, it follows that~$tc^F\geq tc^f$.

 \begin{prop}\label{prop - fill with transverse intersections}
 Given a pair of maps~$(F,f)$ as in Equation~(\ref{eq-pair of maps}), there exists a pair of maps~$(G,g)$ with~$g$ homotopic to~$f$, such that~$tc^g=tc^G=0$. 
 \begin{proof}
 We again use~\cite[Theorem 2.1]{Hirsch} to replace separating systems with tangential intersections with separating systems that intersect transversely.
 
 {\bf Step 1}: If~$tc^{f}\geq 1$ we show that we can homotope~$f$ to a map~$f_1$ and replace~$F$ with~$F_1$ such that~$(F_1,f_1)$ are a pair, and~$tc^{f_1}< tc^f$.
 
Let $\S_p\in \im(f)_0$ such that~$f(p)=\Sigma_p$ and~$tc_f(\Sigma_p)=1$, \emph{i.e.,}~$\Sigma_p$ has a non-transverse intersection with some $\Sigma_q \in \im(f)_0$ corresponding to~$q\in (D^{k+1})^{(0)}$. Then since~$f$ maps~$S^k$ into~$ |\Sep(\rho)^\delta_\bt|$,~$\Sigma_p$ is disjoint from all~$\Sigma_r\in \im(f)_0$ such that~$r$ is a vertex in $\lk_{ S^k}(p)$. Being disjoint is an open condition in~$\ESS(L)$, so there is an open neighbourhood of~$\Sigma_p$ in~$\ESS(L)$ of separating systems which remain disjoint from the~$\S_r$ and from~$\rho$. By residuality of transverse embeddings, we can find a~$\Sigma'_p$ in this neighbourhood transverse to all separating systems in~$\im(F)_0$.
 
 Now we create a new triangulation of~$D^{k+1}$ by adding a vertex~$p'$ and the simplices of~$p'* \st_{S^k}(p)$. Let~$F_{1}$ be a simplicial map from this new triangulation of~$D^{k+1}$ to~$|\TSeprho_\bt|$ defined as follows. On~0-simplices let~$F_1$ agree with~$F$ away from~$p'$, and~$F_{1}(p')=\Sigma'_p$. Since~$\Sigma'_p$ is transverse to $\im(F)_0$ and disjoint from~$\Sigma_r$ when $r\in \lk_{ S^k}(p)$, we can extend this map over the simplices of the new triangulation to get a pair of maps~$(F_1,f_1)$. (The map $f_1:S^k\to |\Sep(\rho)^\delta_\bt|$ is given by~$F_{1}$ restricted to~$\p D^{k+1}$.) It remains to show that~$f_1$ is homotopic to~$f$. We achieve this via the homotopy~$h:([0,1]\times~S^{k})\to |\Sep(\rho)_\bt|$ with~$h_0=f$,
 ~$h_1=f_1$ and such that as~$t\in[0,1]$ varies from~$0$ to~$1$ the weight from the separating system~$\Sigma_p$ shifts to~$\Sigma'_p$. Formally~$h_t=h_0$ away from~$\st_{S^k}(p)$, $h_t(p)=[(\S_p,\S'_p),(1-t,t)]$ and, for a $k$-simplex containing~$p$ in~$\st_{S^k}(p)$,~$h_t$ linearly shifts the weight~$t_i$ attributed to~$\Sigma_p$ when $t=0$, to be attributed to~$\Sigma'_p$ when $t=1$. 
 
 A schematic of this homotopy is depicted in the Figure \ref{fig:WeightShiftHomotopy}, when~$k=1$.
 \begin{figure}[h!]
     \centering
     \begin{tikzpicture}{scale=0.6}
     \coordinate (A) at (.8,-.3);
     \coordinate (B) at (0,0);
     \filldraw[fill=champagne!60] (-1.75,2.5)--(-.5,1.5)--(A)--(-.5,-1.5)--(-1.75,-2.5);
     \filldraw[fill=cyan!40,opacity=0.7] (B)--(A)-- (-.5,1.5); 
     \filldraw[fill=cyan!40,opacity=0.7] (B)--(A)-- (-.5,-1.5); 
     \filldraw[fill=champagne!60] (-.8,-.2)--(B)-- (-.5,1.5); 
     \filldraw[fill=champagne!60] (-.8,-.2)--(B)-- (-.5,-1.5); 
     \draw (B)--(-.5,1.5);
     \draw (B)--(-.5,-1.5);
     \draw (-.8,-.2)--(-.5,1.5);
     \draw (-.8,-.2)--(-.5,-1.5);
     \draw (-.8,-.2)--(B);
     \draw (-.8,-.2)--(-1.75,.3);
     \draw (-.8,-.2)--(-1.75,-.4);
     \draw (-.5,1.5)--(-1.75,1);
     \draw (-.5,-1.5)--(-1.75,-2);
     \filldraw (B) circle(2pt);
     \filldraw (A) circle(2pt);
     \filldraw (-.5,1.5) circle(2pt);
     \filldraw (-.5,-1.5) circle(2pt);
     \filldraw (-.8,-.2) circle(2pt);
     \node at (-.2,1.5) {$\scriptstyle\Lambda_1$};
     \node at (-.2,-1.5) {$\scriptstyle\Lambda_2$};
     \node at (1.1,-.3) {$\scriptstyle\Sigma_p'$};
     \node at (-.3,.15) {$\scriptstyle\Sigma_p$};

     \end{tikzpicture}
     \caption{Replacing $(F,f)$ with $(F_1,f_1)$ when $k=1$. The image of~$F$ is shown in tan, and the image of~$F_1$ in tan and blue. The link of $\Sigma_p$ in $S^1$ is $\Lambda_1$ and $\Lambda_2$. The homotopy from~$f$ to $f_1$ replaces the star of $\Sigma_p$ in the boundary of $\im(F)$ with the star of $\Sigma_p'$ in the boundary of $\im(F_1)$. It linearly shifts the weight from $\Sigma_p$ to $\Sigma_p'$, which pushes the arc $\Lambda_1-\Sigma_p-\Lambda_2$ across the two rightmost blue triangles to $\Lambda_1-\Sigma_p'-\Lambda_2$.}
     \label{fig:WeightShiftHomotopy}
 \end{figure}
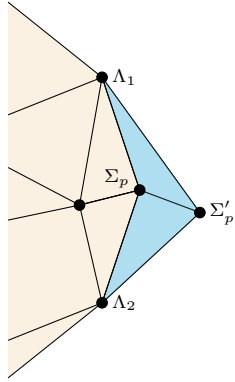

 The outcome is a pair of maps~$(F_{1},f_1)$ such that $tc^{f_1}\leq tc^f-1$, as required (since we have replaced~$\Sigma_p$ satisfying~$tc_f(\S_p)=1$ with~$\S'_p=\S_{p'}$ satisfying~$tc_{f^1}(\S_{p'})=0$. Furthermore, the point~$p$ lies in the interior of~$D^{k+1}$, in the new triangulation.
 
 Iterating Step 1 until no longer possible, we produce a pair~$(F_n,f_n)$ with~$f_n$ homotopic to~$f$ and~$tc^{f_n}=0$.

 {\bf Step 2}: Start with the pair~$(F_n,f_n)$ created in Step 1. We replace~$(F_n,f_n)$ with~$(F_{n+1},f_n)$ such that $tc^{F_{n+1}}< tc^{F_n}$.
 
Consider~$p\in (D^{k+1}\setminus \p D^{k+1})^{(0)}$ such that~$tc_{F_n}(\Sigma_p)=1$. By residuality of transverse embeddings there exists a separating system~$\Sigma'_p$ transverse to all spheres in separating systems of~$\im(F_n)_0\setminus \S_p$. We modify the map~$F_n$ to a map~$F_{n+1}$, by setting~$F_{n+1}(p)=\Sigma'_p$ and for all simplices~$\sigma$ in~$\st_{D^{k+1}}(p)$ replacing the occurrence of~$\Sigma_p$ in~$F_n(\sigma)$ with~$\Sigma_p'$ of the same weight. Then~$F_{n+1}=f_n$ on~$S^k=\p D^{k+1}$, and $tc^{F_{n+1}}\leq tc^{F_n}-1$ as required.
 
 After a finite number of iterations of Step 2, we produce a pair~$(F_m,f_n)=(G,g)$ such that $tc^g=tc^G=0$, as required.
 \end{proof}
 \end{prop}

We now consider a pair of maps~$(G,g)$ with $tc^g=tc^G=0$ as produced by the previous step, \emph{i.e.,}~such that all separating systems in~$\im(G)$ intersect transversely, regardless of whether the intersections are ghost or real. 

Let~$S_p$ and~$S_q$ be transversely embedded spheres in~$\R^3$, which intersect in a disjoint union of circles. Let~$D$ be a disk in~$S_p$ with boundary one of these circles, such that the interior of~$D$ contains no other intersection circle with~$S_q$. We recall what we mean by~\emph{surgery} on $S_q$ with respect to the disk~$D\subset~S_p$. Let~$(-1,1)\times S_q$ and $(-1,1)\times S_p$ be tubular neighbourhoods of~$S_p$ and~$S_q$, such that~$\{-\epsilon\} \times S_p$ lies in the interior of~$S_p$ when~$\epsilon>0$ and likewise for~$S_q$, and such that~$\{\epsilon\} \times~S_p$ intersects~$\{\delta\}\times S_q$ transversely for all~$\epsilon, \delta \in (-1,1)$, with intersection pattern the same as that of~$S_p$ and~$S_q$. These neighbourhoods can always be found with such properties, since immersions are open in~$\uemb(S^2,\R^3)\times \uemb(S^2,\R^3)$. Restricting the tubular neighbourhood of $S_p$ to $D$ yields a product region~$(-1,1)\times D$ containing~$D$. 

For small $\delta \in (0,1]$,  the intersection  $((-1,1)\times D)\cap (\{-\delta\} \times S_q)$ is an annulus $A\subset \{-\delta\}\times S_q$. Removing the interior of $A$ from $\{-\delta\} \times S_q$ yields two disks~$D^+_q(-\delta)$ and~$ D_q^-(-\delta)$, such that $\{-\delta\} \times S^q= D^+_q(-\delta) \cup {A} \cup D^-_q(-\delta)$. Then by \emph{surgery on~$S_q$ with respect to~$D$} we mean replacing~$S_q$ with the two spheres~$S_q^-$ and~$S_q^+$ where

$$S_q^-=D^-_q(-\delta)\cup (\{-1\} \times D)\text{  and  }S_q^+=D^+_q(-\delta)\cup (\{1\} \times D).$$ 

These are manifolds with corners and therefore not elements of $\uemb(S^2,\R^3)$, but we can smooth the corners in a neighbourhood arbitrarily close to the corners.  We do this smoothing and by abuse of notation call the resulting smoothly embedded spheres~$S_q^+$ and $S_q^-$.
Note that by construction,~$S_q^+$ and~$S_q^-$ are disjoint from~$S_q$, and lie in the interior of the ball bounded by~$S^q$. A schematic of the surgery is shown in Figure~\ref{fig:Surgery Diagram}.
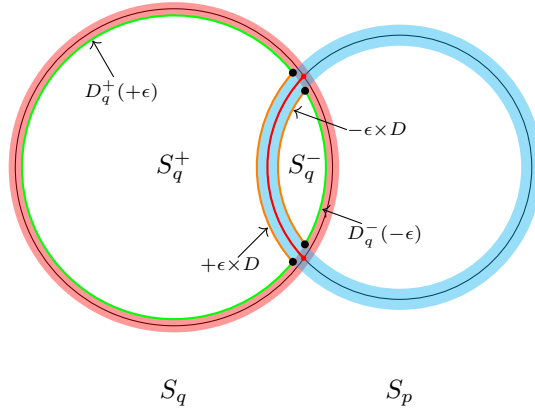
\begin{figure}[h!]
    \centering
    \begin{tikzpicture}{scale=0.6}
     \draw (0,0) circle(60pt);
     \draw (3,0) circle(50pt);
     \fill [red,opacity=0.4,even odd rule] (0,0) circle[radius=62.5pt] circle[radius=57.5pt];
     \fill [cyan,opacity=0.4,even odd rule] (3,0) circle[radius=54pt] circle[radius=46pt];
   
    \draw[thick,green] ([shift=(-30.5:57.5pt)]0,0) arc (-30.5:30.5:57.5pt);
    \draw[thick,green] ([shift=(38.5:57.5pt)]0,0) arc (38.5:321.5:57.5pt);

   \draw[thick,red] ([shift=(136.5:50pt)]3,0) arc (136.5:223.5:50pt);

    \draw[thick,orange] ([shift=(138.5:54pt)]3,0) arc (138.5:222.5:54pt);
    \draw[thick,orange] ([shift=(141:46pt)]3,0) arc (141:219:46pt);
    
    \fill ([shift=(-30.5:57.5pt)]0,0) circle (1.5pt);
    \fill ([shift=(321.5:57.5pt)]0,0) circle (1.5pt);
    \fill ([shift=(38.5:57.5pt)]0,0) circle (1.5pt);
    \fill ([shift=(141:46pt)]3,0) circle (1.5pt);
    \fill[red] ([shift=(136.5:50pt)]3,0) circle (1pt);
    \fill[red] ([shift=(223.5:50pt)]3,0) circle (1pt);

    \draw[->](-.8,1.2)--++(120:16pt);
    \draw[->](.85,-1.2)--++(45:16pt);
    \draw[->](2.3,.5)--++(160:21.5pt);
    \draw[->](2.5,-.75)--++(160:16pt);

    \node at (0,-3) {$S_q$};
    \node at (3,-3) {$S_p$};
    
    \node at (0,0) {$S_q^+$};
    \node at (1.75,0) {$S_q^-$};
    \node at (-.7,1) {$\scriptstyle D_q^+(-\delta)$};
    \node at (.75,-1.3) {$\scriptstyle 1\times D$};
    \node at (2.7,.5){$\scriptstyle -1\times D$};
     \node at (2.8,-.9){$\scriptstyle D_q^-(-\delta)$};

    \end{tikzpicture}

    \caption{A schematic picture of surgery on $S_q$ with respect to $D$. $S_p,~S_q$ are represented by the black circles on the left and right respectively, and $D$ is the arc in red on $S_p$.  The $\delta$-neighbourhood of $S_q$ is light red while the $\epsilon$-neighbourhood of $S_p$ is shown in light blue. The surgered sphere $S_q^+$ is the union of the green and orange arc in the interior of $S_q$, while $S_q^-$ is the union of the orange and green arc in the interior of $S_p$.}
    \label{fig:Surgery Diagram}
\end{figure}

\begin{defn}\label{defn-Disk(F)}
Given a pair of maps~$(G,g)$, with~$tc^G=0$, consider the image of all separating systems in~$\im(G)_0$. For each intersection circle in this image, the two intersecting spheres are divided by this circle into two disks. Let~$\Disk(G)$ be the set of all disks arising in this way in~$\im(G)_0$. A \emph{minimal disk} in~$\Disk(G)$  is any disk which has the smallest area, where area is calculated using the Euclidean metric on~$\R^3$.
\end{defn}

\begin{proof}[Proof of Theorem~\ref{thm-discrete fiber is contractible}] We show that $\iota_p:\Sep(\rho)^\delta_p\hookrightarrow \TSeprho_p$ is  a homotopy equivalence, by showing the relative homotopy groups vanish. We start with a pair of maps~$(F,f)$, \emph{i.e.,}
 \[ \xymatrix{f: S^{k}\ar[r]\ar@{^{(}->}[d]_{\partial}& |\Sep(\rho)^\delta_\bt|\ar@{^{(}->}[d]^{|\iota_\bt|}\\
 F: D^{k+1}\ar[r]& |\TSeprho_\bt|
 }\]
 
 By Proposition~\ref{prop - fill with transverse intersections} we can assume that the staring pair~$(F,f)$ satisfy the property that that all separating systems in~$\im(F)$ intersect transversely. 
 Let~$\In(F)$ be the total number of intersection circles between separating system of~$\im(F)_0$, and $\In(f)\leq \In(F)$ be the count of intersection circles where at least one of the intersecting spheres lies in some~$\S\in \im(f)_0$. We will produce a pair of maps~$(F',f')$ such that~$f\simeq f'$ and~$\im(F')\subset |\Sep(\rho)^{\delta}_\bt|$, \emph{i.e.,}~$\In(F')$=0.

Consider the sequence~$\Disk(F)$ from Definition~\ref{defn-Disk(F)}, and let~$D_{min}(F)\in \Disk(F)$ be minimal. Then~$D_{min}(F)$ lies in some sphere~$S_p\in \Sigma_p$ for~$\Sigma_p\in \im(F)_0$ the image of~$p\in (D^{k+1})^{(0)}$. Moreover, the circle boundary of~$D_{min}(F)$ is an intersection circle of~$S_p$ with some~$S_q\in \Sigma_q$ for~$\Sigma_q\in \im(F)_0$ the image of some~$q\in (D^{k+1})^{(0)}$. 

For~$j=p$ or $q$, we choose the tubular neighbourhoods~$(-1,1)\times \S_j$ such that for all~$t$,~$\{t\}\times \S_j$ is an essential separating system for~$\rho$ and has the same intersection pattern with~$\im(F)_0\setminus \S_j$ as~$S_j$ does, preserving transversality of intersections. 

We surger~$S_q$ with respect to~$D_{min}(F)$, as detailed in the discussion preceding this proof, using the tubular neighbourhoods~$(-1,1)\times S_p$ and $(-1,1)\times S_q$ given by restricting the tubular neighbourhoods of~$\S_p$ and~$\S_q$ to the spheres~$S_p$ and~$S_q$.
The output of the surgery is two spheres~$S^+_q$ and~$S^-_q$, which by construction are disjoint from~$\rho$ and lie in the interior of the ball bounded by~$S_q$ (recall they arise as disks in~$\{-\delta\} \times S_q$).

Since~$\Sigma_q$ is an essential separating system, there is at most one link piece in each connected component of~$\R^3\setminus \S_q$. Suppose there is a link piece in the connected component whose boundary contains~$S_q$ and which lies in the interior of the ball bounded by~$S_q$. Then one of~$S^+_q$ and~$S^-_q$ bounds a ball containing this link piece -- call this~$S'_q$, and let~$\Sigma'_q$ be the separating system~$(\{-\delta\} \times (\Sigma_q\setminus S_q)) \cup S'_q$. If there is no link piece in the component of~$\R^3\setminus \S_q$ with boundary~$S_q$, then let $\Sigma'_q$ be the separating system~$\{-\delta\} \times (\S_q\setminus S_q)$.

Similar to the proof of Proposition~\ref{prop - fill with transverse intersections} we now have two cases, depending on whether~$q$ is in the boundary or interior of~$D^{k+1}$. In the first case we homotope~$f$ and replace~$F$, changing the triangulation of~$D^{k+1}$ in the process. In the second, we replace~$F$ with a map that sends~$q$ to a different vertex. Unlike the proof of Proposition~\ref{prop - fill with transverse intersections}, the order in which we do these relies on reevaluating~$D_{min}$ after every step, so we cannot first fix the boundary and  then the interior of~$D^{k+1}$. (Note that these two cases encompass the possibility in which~$S_q^+$ or~$S_q^-$ may be parallel to another sphere of $\Sigma_q$, so by construction~$\S'_q$ does not contain parallel spheres).

{\bf Case 1:} $q\in \p D^{k+1}$. Proceed as in the proof of Proposition~\ref{prop - fill with transverse intersections}. Create a new triangulation of~$D^{k+1}$ by adding a vertex~$q'$ and the simplices of~$q'* \st_{S^k}(q)$. Let~$F_{1}$ map from this new triangulation of~$D^{k+1}$ to~$|\TSeprho_\bt|$. On~0-simplices define~$F_1$ to agree with~$F$ away from~$q'$, and~$F_{1}(q')=\Sigma'_q$. Since~$\Sigma'_q$ is disjoint from~$\Sigma_r$ when $r\in \st_{S^k}(q)$, we can extend this map over the simplices of the new triangulation to get a pair of maps~$(F_1,f_1)$. (The map $f_1:S^k\to |\Sep(\rho)^\delta_\bt|$ is given by~$F_{1}$ restricted to~$\p D^{k+1}$.) It remains to show that~$f_1$ is homotopic to~$f$. We achieve this via the homotopy~$h:([0,1]\times~S^{k})\to |\Sep(\rho)_\bt|$ with~$h_0=f$,
 ~$h_1=f_1$ and such that as~$t$ varies from~$0$ to~$1$ the weight from the separating system~$\Sigma_q$ shifts to~$\Sigma'_q$ in~$\im(h_t)$.

{\bf Case 2:} $q\in D^{k+1}\setminus \p D^{k+1}$. Proceed as in the proof of Proposition~\ref{prop - fill with transverse intersections}. Modify the map~$F$ to a map~$F_{1}$, by setting~$F_{1}(q)=\Sigma'_q$ and for all simplices~$\sigma$ in~$\st_{D^{k+1}}(q)$ replacing the occurrence of~$\Sigma_q$ in~$F(\sigma)$ with~$\Sigma_q'$ of the same weight. Then~$F_1=f$ on~$\p D^{k+1}$ -- rename~$f$ to~$f_1$.

In either case, we have replaced the pair~$(F,f)$ with a pair~$(F_1,f_1)$ such that either:
\begin{itemize}
    \item[(i)] $\In(f_1)<\In(f)$. Note that it is not necessarily the case that~$\In(F_1)\leq \In(F)$, since we added an extra vertex.
    \item[(ii)] $\In(F_1)<\In(F)$. The reason we created no new intersections when doing our surgery was due to our choice of minimal area disk.
\end{itemize}

We now iterate the process with the pair~$(F_1,f_1)$, considering a minimal area disk in the sequence~$\Disk(F_1)$. At each stage we either reduce the total number of intersections, or we possibly increase it, but reduce the number of total intersections involving separating systems in the image of~$S^k$. Therefore, after a finite number of iterations we produce a pair of maps~$(F',f')$ satisfying~$\In(F')=0$ and we are done. 

Note that we only need to remove real intersections, but by considering both ghost and real intersections we ensure that no new intersections are created when doing the surgery (due to minimality of the area of the surgery disk).
\end{proof}

Thus we have finished the proof of Theorem~\ref{abcthm - sep equiv}.

\appendix

\section{Further implications of Theorem~\ref{abcthm - sep equiv}}\label{section - implications for homotopy groups}
In this appendix we introduce a levelwise `round' subspace of~$\Sep$, and use the previous sections to obtain nice representatives for classes in the homotopy groups of~$\link(L)$.

Let $\uremb(\sqcup_{i=1}^kS^2,\R^3)$ be the space of unparametrised embeddings of $\sqcup_{i=1}^kS^2$ into $\R^3$, for which the image of each sphere is \emph{round}, \emph{i.e.,} it bounds a Euclidean ball in~$\R^3$. The following result is Lemma 4.2 of \cite{BrendleHatcher}, but we include a proof here for the reader's convenience.

\begin{lem}\label{lem - rounding spheres is he}  The map  $\iota \colon \uremb(\sqcup_{i=1}^kS^2,\R^3)\hookrightarrow \uemb(\sqcup_{i=1}^kS^2,\R^3)$ is a homotopy equivalence.  
\begin{proof} We will prove the corresponding statement for parametrised embeddings, \emph{i.e.} that the inclusion $\remb(\sqcup_{i=1}^kS^2,\R^3)\rightarrow\emb(\sqcup_{i=1}^kS^2,\R^3)$ is a homotopy equivalence.  The statement will then follow from Smale's theorem that $\Diff(S^2)\simeq O(3)$. To simplify notation, we will abbreviate $\emb(\sqcup_{i=1}^kS^2,\R^3)$ to $\emb$ and  $\remb(\sqcup_{i=1}^kS^2,\R^3)$ to $\remb$.  We will show that the relative homotopy groups $\pi_k(\emb,\remb)$ vanish for all $k\geq 1$. Since both of these spaces have the homotopy type of CW complexes, the lemma will follow.

Consider a relative homotopy class $\xi\colon (D^k,\partial D^k)\rightarrow (\emb,\remb)$. Since $D^k$ is contractible, we may assume the spheres are labelled over all of $D^k$.

The configuration of spheres have the same combinatorial type over the whole disk $D^k$ (see Definition~\ref{def-combinatorial type}). In particular there is a well-defined notion of nesting of spheres, as well as which spheres are outermost. We will homotope $\xi$ into $\remb$ first by making the outermost spheres round, then the second outermost spheres, etc. Let $S^2_1,\ldots, S^2_{n_1}$ be the collection of outermost spheres.  Use isotopy  extension to get a family of embeddings $\hat{\xi}\colon D^k\rightarrow \emb(\sqcup_{i=1}^{n_1}B_i^3,\R^3)$, where $B^3_i$ bounds $S^2_i$.  
By the proof of the Smale conjecture \cite{Hatcher} that $\Diff(B^3\text{ rel } \partial B^3)$ is contractible, we can homotope $\hat{\xi}$ so that on $\partial D^k$ it consists of round embeddings (\emph{i.e.}~rescaled isometric embeddings) of $\sqcup_{i=1}^{n_1}B_i^3$.

Denote the subspace of round  embeddings of $\sqcup_{i=1}^{n_1}B_i^3$ by $
\remb(\sqcup_{i=1}^{n_1}B_i^3,\R^3)$. 
Thus $\hat{\xi}$ defines a relative homotopy class of   $(\emb(\sqcup_{i=1}^{n_1}B_i^3,\R^3),\remb(\sqcup_{i=1}^{n_1}B_i^3,\R^3))$. We now use the linearisation procedure as in the proof of Lemma \ref{lem:Balls Complement} to show this class is trivial. Explicitly, given embeddings $f_i\colon B_i^3\rightarrow \R^3$, for each $t\in [0,1]$ define a family of rescaled embeddings \[f_i(x)\mapsto \frac{1}{1-t}f_i((1-t)x).\] 
As $t\to 1$ this replaces $f_i(x)$ with the linear map $D_0(f_i)(x)$, where $D_0(f_i)$ is the derivative of $f_i$ at the origin, and continuity holds at $t=1$ because $f_i$ is differentiable. This replaces each $f_i$ with a linear embedding, demonstrating that $\emb(B^3,\R^3)\simeq \GL_3(\R)\times \R^3$, where the second factor corresponds to the image of the origin. Since $f_i|{\p D^k}$ is already linear, this process does not change $\hat{\xi}|_{\p D^k}$ for all $t\in[0,1]$. Now we use $\GL_3(\R)$ is homotopy equivalent to the orthogonal group $O(3)$, so we may further replace each $f_i$ by a linear isometric embedding. 

The linear isometric embeddings obtained at the end of this process have the same centerpoints as the initial $f_i$, hence we must rescale the balls so that throughout this process the images of the $B_i^3$ remain disjoint.  We can choose a rescaling function that depends continuously on the minimum distance between any two of the centers $\{f_i(0)\}$, and that continuously damps out to 1 as we approach $\p D^k$. At the end of this homotopy, we have a $D^k$-family of round embeddings of $B_1^3,\ldots, B_{n_1}^3$.  We then proceed to the next outermost spheres, which will have moved by isotopy extension during this process, and repeat the same procedure until all the spheres are round on all of $D^k$.\qedhere
\end{proof}
\end{lem}

\begin{defn}\label{defn-SREmb}
We define the semi-simplicial space~$\RSep_\bt$ by setting~$\RSep_p$ to be the levelwise subspace of~$\Sep_p$ consisting of embeddings such that the image of each embedded sphere in~$\R^3$ is \emph{round}, that is it bounds a Euclidean ball in~$\R^3$.
\end{defn}

\begin{prop}\label{prop-SRemb he Semb}
The inclusion map~$\iota_\bt:\RSep_\bt\hookrightarrow \Sep_\bt$ induced by the levelwise inclusion of embedding spaces, induces a homotopy equivalence on geometric realisations.
\begin{proof}
From~\cite[Lemma 2.4]{EbertRandalWilliams}, it is enough to show that
\[
\iota_p:\RSep_p\hookrightarrow \Sep_p
\]
is a homotopy equivalence. A point in $\Sep_p$ may be regarded as an embedding $\rho\in \link(L)$ together with a $(p+1)$-tuple $(\Sigma_0,\ldots,\Sigma_p)$, where each $\Sigma_i$ is a separating system for $\rho$ and all the $\Sigma_i$ can be embedded disjointly. Consider the map \[\Sep_p\rightarrow\bigsqcup_{k=1}^{\infty}\uemb(\sqcup_{i=1}^kS^2,\R^3)\] 
which sends each point in~$\Sep_p$ to the union $\sqcup_{i=0}^p\Sigma_i$ of corresponding separating systems. This is a fibration over each connected component, hence it has the homotopy lifting property. We lift the deformation retraction from Lemma \ref{lem - rounding spheres is he} to get a homotopy equivalence $\RSep_p\simeq\Sep_p$.
\end{proof}
\end{prop}

There is an augmentation map $\varepsilon_\bt:\RSep_\bt\to \link(L)$ inherited from that of~$\Sep_\bt$.

\begin{thm}\label{thm - SRemb and emb equivalent}
The realisation of the augmentation map
\[|\varepsilon_\bt|:|\RSep_\bt|\to \link(L) \]is a homotopy equivalence.
\begin{proof}
From Proposition~\ref{prop-SRemb he Semb} we have that the map induced by inclusion
\[ |\RSep_\bt|\overset{|\iota_\bt|}{\hookrightarrow}|\Sep_\bt|
\]
is a weak equivalence, and from Theorem~\ref{abcthm - sep equiv}, the realisation of the augmentation map
\[|\Sep_\bt|\overset{|\varepsilon_\bt|}{\rightarrow} \link(L)\]
is also weak equivalence. Since all the spaces involved have the homotopy type of CW-complexes, we conclude that the composition of these two maps is a homotopy equivalence.
\end{proof}
\end{thm}

\begin{defn}
Suppose there exists a map~$f:S^k\to \link(L)$ where~$S^k$ is triangulated such that:
\begin{enumerate}[(i)]
    \item on the image of each simplex there exists a separating system which varies by isotopy as the embedded link does, and,
    \item on shared faces of simplices all separating systems exist simultaneously and disjointly.
\end{enumerate}
Then we will say the image of the triangulated sphere~$S^k$ under such a map~$f$ exhibits a \emph{compatible separating triangulation} in~$\link(L)$. If in addition every separating sphere is round, then we say the map~$f$ exhibits a \emph{compatible round separating triangulation} in~$\link(L)$.
\end{defn}

\begin{cor}\label{cor - compatible round separating triangulation}
Each homotopy class in~$\pi_k(\link(L))$ has a representative~$f:S^k\to \link(L)$ such that~$f$ exhibits a compatible round separating triangulation.
\begin{proof}
From Theorem \ref{thm - SRemb and emb equivalent},~$\pi_k(\link(L))\simeq \pi_k(|\RSep_\bt|)$, and so every map~$f:S^k\to \link(L)$ lifts to a map $\hat{f}:S^k\to |\RSep_\bt|$ and we consider the image~$\hat{f}(S^k)$ of this map. Each point in the image is a point in~$|\RSep_\bt|$ and, as discussed in Section~\ref{section - contractible}, the data associated to this point is a link, and a collection of disjoint separating systems~$\Sigma_0,\ldots,\Sigma_p$ for this link, such that each separating system has a {weight} (corresponding to the barycentric coordinates of the simplex direction of the geometric realisation). For each point~$x\in S^k$, we choose a separating system~$\Sigma(x)$ that has non zero weight at~$\hat{f}(x)$. Then since having non-zero weight is an open condition, there exists an open set~$U_x\subset S^k$ containing~$x$ for which~$\Sigma_x$ remains a separating system with non zero weight in~$\hat{f}(U_x)$. The~$U_x$ give an open cover of~$S^k$ and we choose a finite subcover~$\bigcup_{\alpha \in \Lambda}U_\alpha$. If~$U_\alpha\cap U_\beta\neq \emptyset$ for~$\alpha, \beta \in \Lambda$, it follows that the separating systems~$\Sigma_\alpha$  and~$\Sigma_\beta$ both have non-zero weight on the intersection, which implies that the separating systems can be realised disjointly. 
We now triangulate~$S^k$ so that each simplex lies fully in some~$U_\alpha$, and this completes the proof.
\end{proof}
\end{cor}

\begin{rem}
    By Proposition~\ref{prop-SRemb he Semb} we can also deduce the weaker claim that each homotopy class in~$\pi_k(\link(L))$ has a representative~$f:S^k\to \link(L)$ such that~$f$ exhibits a compatible separating triangulation, where the spheres of the separating systems are not required to be round. However, round spheres behave like points when shrunk, and also allow applications once one knows the homotopy type of a link piece in~$\R^3$ (for example in Brendle--Hatcher the round spheres allow one to round the individual unknots \cite{BrendleHatcher}).
    The existence of such a compatible separating triangulation for a disk as opposed to a sphere (which follows the same proof as above) is a key point in Step 2 of the proof of Theorem 4.1 in Brendle--Hatcher \cite{BrendleHatcher}, in which the existence is stated but not fully justified. 
\end{rem}
\bibliography{mybib}{}
\bibliographystyle{alpha}
\end{document}